\documentclass[a4paper,leq,11pt]{article}

\usepackage{amsmath,amssymb}
\usepackage{amsthm} 
\usepackage{graphicx}
\usepackage{color}
\usepackage{paralist}
\usepackage{hyperref,url} % For hyperlinks in the PDF
\usepackage{booktabs} % Horizontal rules in tables
\usepackage[hmarginratio=1:1, top=32mm, columnsep=20pt, margin=1in]{geometry} % Document margins

\theoremstyle{plain}
\newtheorem{theorem}{Theorem}[section]

\newtheorem{lemma}{Lemma}[section]

\theoremstyle{definition}

\theoremstyle{remark}
\newtheorem{remark}{Remark}[section]

\newcommand{\Id}{\mathbb{I}}
\newcommand{\R}{\mathbb{R}}

\newcommand{\dom}[1]{\mathrm{dom}(#1)}
\newcommand{\ran}[1]{\mathrm{ran}(#1)}
\newcommand{\gra}[1]{\mathrm{gra}(#1)}
\newcommand{\set}[1]{\left\{#1\right\}}
\newcommand{\sets}[1]{\{#1\}}

\newcommand{\norms}[1]{\Vert#1\Vert}

\newcommand{\iprods}[1]{\langle #1\rangle}

\newcommand{\argmin}{\mathrm{arg}\!\displaystyle\min}

\newcommand{\BigO}[1]{\mathcal{O}\left(#1\right)}
\newcommand{\BigOs}[1]{\mathcal{O}\big(#1\big)}

\newcommand{\SmallO}[1]{o\left(#1\right)}

\newcommand{\Eproof}{\hfill $\square$}

\newcommand{\Xc}{\mathcal{X}}

\newcommand{\Tc}{\mathcal{T}}
\newcommand{\Nc}{\mathcal{N}}
\newcommand{\Pc}{\mathcal{P}}
\newcommand{\Gc}{\mathcal{G}}
\newcommand{\Uc}{\mathcal{U}}
\newcommand{\Vc}{\mathcal{V}}

\newcommand{\zer}[1]{\mathrm{zer}(#1)}

\newcommand{\mcal}[1]{\mathcal{#1}}

\newcommand{\revised}[1]{{\color{black}#1}} %% After revision.

\newcommand{\beforesec}{\vspace{-3ex}}
\newcommand{\aftersec}{\vspace{-2ex}}
\newcommand{\beforesubsec}{\vspace{-3ex}}
\newcommand{\aftersubsec}{\vspace{-2ex}}

\title{Extragradient-Type Methods with $\BigO{1/k}$ Last-Iterate Convergence Rates for Co-Hypomonotone Inclusions}
\author{Quoc Tran-Dinh\vspace{0.25ex}\\
\newline {Department of Statistics and Operations Research}\\
\newline The University of North Carolina at Chapel Hill\\
318 Hanes Hall, UNC-Chapel Hill, NC 27599-3260.\\
\newline \textit{Email:} \url{quoctd@email.unc.edu}.}

\date{Version 1: 2023/2/8. This is Version 2: 2023/10/14}

\begin{document}
%
%\author{Quoc Tran-Dinh \\
%%
%%\institute{Quoc Tran-Dinh\\
%              Department of Statistics and Operations Research\\
%              The University of North Carolina at Chapel Hill, 318 Hanes Hall, Chapel Hill, NC 27599-3260\\
%              \email{quoctd@email.unc.edu}
%%\date{Version 3: The first version  was posted on Arxiv on March 13}
\maketitle

%%%%%%%%%%%%%%%%%%%%%%%%%%%%%%%
%%%%%%%%%%%%% ABSTRACT %%%%%%%%%%%
%%%%%%%%%%%%%%%%%%%%%%%%%%%%%%%
\begin{abstract}
\normalsize
We develop two ``Nesterov's accelerated'' variants of the well-known extragradient method to approximate a solution of a co-hypomonotone inclusion constituted by the sum of two operators, where one is Lipschitz continuous and the other is possibly multivalued.
\revised{The first scheme can be viewed as an accelerated variant of Tseng's forward-backward-forward splitting (FBFS) method, while the second one is a Nesterov's accelerated variant of the ``past'' FBFS scheme, which requires only one evaluation of the Lipschitz operator and one resolvent of the multivalued mapping.}
Under appropriate conditions on the  parameters, we theoretically prove that both algorithms achieve $\BigO{1/k}$ last-iterate convergence rates on the residual norm, where $k$ is the iteration counter.
Our results can be viewed as alternatives of a recent class of Halpern-type methods for root-finding problems.
For comparison, we also provide a new convergence analysis of the two recent extra-anchored gradient-type methods for solving co-hypomonotone inclusions.

\vspace{1ex}

\vspace{1ex}
\end{abstract}

%%%%%%%%%%%%%%%%%%%%%%%%%%%%%%%%%%%%%%%%%%%
%%% 1. Introduction.
%%%%%%%%%%%%%%%%%%%%%%%%%%%%%%%%%%%%%%%%%%%
\beforesec
\section{Introduction}\label{sec:intro}
\aftersec
The goal of this paper is to extend two recent ``Halpern's accelerated'' extragradient-type methods for equation root-finding problems in \cite{lee2021fast,tran2021halpern,yoon2021accelerated} to inclusion root-finding settings under a so-called ``co-hypomonotone'' assumption \cite{bauschke2020generalized}.
Unlike recent works along this line \cite{cai2022accelerated,cai2022baccelerated}, we exploit a different view from \cite{tran2022connection}, which establishes the relation between Halpern's fixed-point iteration \cite{halpern1967fixed} with ``optimal'' convergence rates  \cite{diakonikolas2020halpern,lee2021fast,lieder2021convergence,yoon2021accelerated} and Nesterov's accelerated techniques widely extended to root-finding problems, e.g.,  in  \cite{attouch2020convergence,chen2017accelerated,kim2021accelerated,mainge2021fast,Nesterov1983,Nesterov2004}.
To set our stage, we consider the following operator inclusion (also called a generalized equation \cite{Rockafellar2004}):
\begin{equation}\label{eq:coMI}
\textrm{Find $x^{\star} \in \R^p$ such that:} \quad 0 \in Fx^{\star} + Tx^{\star},
\tag{chMI}
\end{equation}
where $F$ is a single-valued operator from $\R^p\to\R^p$, and $T : \R^p \rightrightarrows 2^{\R^p}$ is a possibly multivalued operator, where $2^{\R^p}$ is the set of all subsets of $\R^p$.
Up to this point, we have not made any assumption on \eqref{eq:coMI}.
In particular, if $T = 0$, then \eqref{eq:coMI} reduces to the following equation root-finding problem:
\begin{equation}\label{eq:coME}
\textrm{Find $x^{\star} \in \R^p$ such that:} \quad Fx^{\star} = 0,
\tag{chME}
\end{equation}
of a [non]linear operator $F$.
Alternatively, if $T := \Nc_{\mcal{C}}$, the normal cone of a nonempty, closed, and convex set $\mcal{C}$, then \eqref{eq:coMI} reduces to a variational inequality problem (VIP), widely studied in the literature, see for examples  \cite{Facchinei2003,harker1990finite,Konnov2001,marcotte1991application}.
For our convenience, we denote $G := F + T$ and assume that the solution set $\zer{G} := G^{-1}(0) = \sets{x^{\star} \in \R^p : 0 \in Gx^{\star} = Fx^{\star} + Tx^{\star} }$ of \eqref{eq:coMI} is nonempty. 

The inclusion \eqref{eq:coME} though looks simple, it covers many fundamental problems in different fields by appropriately reformulating them into \eqref{eq:coMI}.
These problems include optimization (both unconstrained and constrained settings), minimax optimization, two-person game, variational inequality and its special cases, and more generally, fixed-point problems, see \cite{Bauschke2011,reginaset2008,Facchinei2003,phelps2009convex,Rockafellar2004,Rockafellar1976b,ryu2016primer}.
Alternatively, \eqref{eq:coMI} is also ubiquitous in large-scale modern machine learning and data science applications, especially in stochastic settings, see, e.g., \cite{Bottou2018,lan2020first,sra2012optimization}.
However, most existing works often focus on special cases of \eqref{eq:coME} such as optimization, convex-concave minimax, and supervised learning models as can be found in 
\cite{censor2011subgradient,Facchinei2003,Konnov2001,korpelevich1976extragradient,malitsky2015projected,malitsky2019golden,malitsky2014extragradient,Monteiro2011,Nesterov2007a,popov1980modification,solodov1999new,tseng2000modified}.
A recent trend in adversarial machine learning, reinforcement learning, online learning, and distributionally robust optimization has motivated the development of efficient and reliable methods for solving \eqref{eq:coMI}, especially in non-monotone cases \cite{arjovsky2017wasserstein,azar2017minimax,bhatia2020online,goodfellow2014generative,jabbar2021survey,levy2020large,lin2022distributionally,madry2018towards,rahimian2019distributionally}.

Solution approaches for solving \eqref{eq:coMI} often rely on a fundamental assumption: \textit{monotonicity} of $F$ and $T$, or of $G$.
Several methods generalize existing optimization algorithms such as gradient and proximal-point schemes, and exploit the splitting structure of \eqref{eq:coMI} to use individual operators defined on $F$ and $T$, respectively.
Classical methods include gradient (also known as forward or fixed-point iteration), extragradient, past-extragradient, optimistic gradient,  proximal-point,  forward-backward splitting, forward-backward-forward splitting, Douglas-Rachford splitting, forward-reflected-backward splitting, golden ratio, projective splitting methods, and their variants, see for example \cite{Bauschke2011,cevher2021reflected,Combettes2005,Davis2015,Lions1979,Facchinei2003,malitsky2015projected,malitsky2020forward,popov1980modification,tseng2000modified} for more details.
However, extensions to non-monotone settings remain limited.
\revised{Under the non-monotonicity,  \eqref{eq:coMI} is expected to cover a broader class of problems arising from modern applications in, e.g., adversarial machine learning, reinforcement learning, [distributionally] robust optimization, and game theory.
Recent works on nonmonotone \eqref{eq:coMI} can be found in \cite{bohm2022solving,cai2022accelerated,diakonikolas2021efficient,lee2021fast,luo2022last,pethick2022escaping,pennanen2002local}.}

Our goal in this paper is to develop new schemes for solving \eqref{eq:coMI} under the \textit{co-hypomonotone} structure of $F+T$ (see its definition in Section~\ref{sec:background}).
Therefore, let us first review some relevant works to our methods in this paper.
Classical splitting methods mentioned above for solving \eqref{eq:coMI} or its special cases heavily rely on the monotonicity of $F+T$ (or each of them) and the Lipschitz continuity or the co-coerciveness of $F$, see, e.g., \cite{Bauschke2011,Combettes2005,Davis2015,Lions1979,Facchinei2003,malitsky2015projected,popov1980modification,tseng2000modified}.
For the nonmonotone case, \cite{pennanen2002local} studied local convergence of proximal-point methods, while \cite{combettes2004proximal} utilized the hypomonotone and co-hypomonotone structures. 
Some other works considered minimax problems, special cases of \eqref{eq:coMI}, by imposing weak monotonicity or other appropriate regularity conditions (e.g., two-sided Polyak-{\L}ojasiewicz condition, or interaction dominance) as in \cite{grimmer2023landscape,lin2018solving,yang2020global}.
Recently, convergence rates of extragradient methods for a class of nonmonotone (i.e. weak Minty) variational inequalities were studied in \cite{diakonikolas2021efficient}, and then in \cite{pethick2022escaping} for \eqref{eq:coMI}.
These methods only achieve $\BigOs{1/\sqrt{k}}$ - ``best-iterate" convergence rates, where $k$ is the iteration counter.
The work \cite{luo2022last} proved $\BigOs{1/\sqrt{k}}$ convergence rates on the last iterate for these methods to solve the special case \eqref{eq:coME} of \eqref{eq:coMI} under the co-hypomonotonicity of $F$.
A very recent survey on extragradient-type methods can be found in \cite{tran2023sublinear}.
However, hitherto, establishing the last-iterate convergence rates of classical extragradient-type methods for solving \eqref{eq:coMI} under the nonmonotonicity (e.g., co-hypomonotonicity) remains largely open.

Alternative to classical or ``non-accelerated" algorithms for solving \eqref{eq:coMI}, there is a line of research that develops accelerated methods to solve \eqref{eq:coMI} under both monotone and co-hypomonotone structures.
Early attempts relied on dual averaging techniques such as \cite{Cong2012,Nemirovskii2004,Nesterov2007a}, which require the monotonicity or special assumptions.
Attouch et al \cite{attouch2019convergence} proposed accelerated proximal-point methods to solve \eqref{eq:coMI} under the maximal monotonicity of $G$.
Several following up works have been focussing on Nesterov's acceleration-type methods (i.e. exploiting momentum and possibly correction terms) for solving \eqref{eq:coMI} under monotone assumptions such as \cite{attouch2020convergence,bot2022fast,bot2022bfast,kim2021accelerated,mainge2021fast,mainge2021accelerated}.
Each of these methods can be viewed as a discretization of  an appropriate dynamical system (i.e. a given ordinary differential equation). 
Recently, accelerated methods based on Halpern's fixed-point iteration \cite{halpern1967fixed} have attracted a great attention.
This method was originally developed to approximate a fixed-point of a non-expansive operator, but can be used to solve monotone inclusions of the form \eqref{eq:coMI}.
Perhaps, Sabach and Shtern \cite{sabach2017first} and then Lieder \cite{lieder2021convergence} were the first who proved $\BigO{1/k}$ convergence rates for Halpern's fixed-point iteration by appropriately choosing its parameters.
This method was further exploited to solve variational inequalities in \cite{diakonikolas2020halpern}.
Yoon and Ryu extended Halpern's fixed-point iteration idea to extragradient methods in \cite{yoon2021accelerated} to solve \eqref{eq:coME} that remove the co-coerciveness assumption on $F$.
Lee and Kim \cite{lee2021fast} further advanced \cite{yoon2021accelerated} to develop similar algorithms for solving \eqref{eq:coME} but under the co-hypomonotonicity (a type of non-monotonicity), while do not scarify the $\BigO{1/k}$-convergence rates.
The work \cite{tran2021halpern} exploited the idea of \cite{yoon2021accelerated} and developed a Halpern-type variant for the past-extragradient method in \cite{popov1980modification}, which requires only one operator evaluation per iteration.
Recently, \cite{cai2022accelerated,cai2022baccelerated} extended  \cite{yoon2021accelerated} and \cite{tran2021halpern} to variational inequalities and \eqref{eq:coMI} under either the monotonicity or the co-hypomonotonicity assumption.
Both Halpern's fixed-point iteration and Nesterov's accelerated schemes for root-finding problems are actually related to each other.
Their relation has recently been studied in \cite{tran2022connection} for several schemes.
Other related works along this line  can be found, e.g., in \cite{bohm2022solving,bot2022fast,bot2022bfast,gidel2018variational,gorbunov2022extragradient,gorbunov2022last,lin2018solving,tran2022accelerated}.

%%%% Contribution.
\vspace{1ex}
\textbf{Contribution.}
\revised{In this paper, we propose new ``Nesterov's accelerated variants'' of the forward-backward-forward splitting (FBFS) method  \cite{tseng2000modified} and its past FBFS variant.
The form of our methods is as simple as the one in \cite{tran2022connection} for equation root-finding problems.
Our first algorithm can be viewed as an \textit{accelerated FBFS method}, but abbreviated by AEG as it is rooted from the extragradient method in \cite{korpelevich1976extragradient}.
This algorithm achieves a $\BigO{1/k}$ last-iterate convergence rate on the residual norm of \eqref{eq:coMI}.
However, it requires  two evaluations of $F$ and one evaluation of $J_{\eta T}$ (the resolvent of $\eta T$ for some $\eta > 0$) as often seen in Tseng's FBFS method \cite{tseng2000modified}.
Our second variant is an \textit{accelerated past FBFS scheme} (abbreviated by APEG), which adopts the idea from Popov's method \cite{popov1980modification} or optimistic gradient-type algorithms \cite{daskalakis2018limit}  but using a different past operator value.
This algorithm has almost the same per-iteration complexity as  of the forward-reflected-backward splitting method \cite{malitsky2020forward} or the reflected forward-backward splitting scheme \cite{cevher2021reflected,malitsky2015projected}, but achieves  a $\BigO{1/k}$ last-iterate convergence rate on the residual norm of \eqref{eq:coMI} as in our first algorithm -- AEG.
However, as a compensation, APEG has a smaller stepsize and stricter condition, leading to a worse convergence upper bound than AEG.
Our rates are the same order as  the ones in recent works  \cite{cai2022accelerated,cai2022baccelerated}, but we take a different approach  and develop different algorithms with new convergence analysis.
We believe that this paper is the first showing a possibility to develop Nesterov's accelerated-type methods for nonmonotone (in particular, nonconvex optimization) problems by using different convergence criteria. 
In the second part of this paper, we conduct a new convergence analysis for the two extra-anchored gradient-type methods for solving \eqref{eq:coMI} proposed in \cite{cai2022accelerated,cai2022baccelerated}.
Our analysis can be viewed as a natural extension of the techniques for \eqref{eq:coME} in \cite{lee2021fast,tran2021halpern,yoon2021accelerated} and is different from \cite{cai2022accelerated,cai2022baccelerated}.
Moreover, our second scheme is still different from the one in \cite{cai2022baccelerated}.}

%%%% Paper organization.
\vspace{1ex}
\textbf{Paper organization.}
The rest of this paper is organized as follows.
In Section~\ref{sec:background} we briefly review some related concepts to \eqref{eq:coMI} such as Lipschitz continuity, monotonicity, and co-hypomonotonicity, and recall some preliminary results used in this paper.
Section~\ref{sec:AEG} develops a Nesterov's accelerated extragradient method to solve \eqref{eq:coMI} and analyzes its convergence rate.
Section~\ref{sec:APEG} derives a Nesterov's accelerated variant of the past FBFS method to solve \eqref{eq:coMI} and analyzes its convergence rate.
For completeness and comparison, in Section~\ref{apdx:Halpern_AcEG}, we present an alternative convergence analysis for the Halpern-type extragradient methods (also called extra-anchored gradient-type methods) in  \cite{cai2022accelerated,cai2022baccelerated}.
However, our second variant is significantly different from the one in \cite{cai2022baccelerated}  due to the choice of a past value.

%%%%%%%%%%%%%%%%%%%%%%%%%%%%%%%%%%%%%%%%%%%%%
%%%% 2. Background and Preliminary Results.
%%%%%%%%%%%%%%%%%%%%%%%%%%%%%%%%%%%%%%%%%%%%%
\beforesec
\section{Background and Preliminary Results}\label{sec:background}
\aftersec
We first review some background on monotone operators and related concepts.
Then, we recall Tseng's classical  FBFS method and its past FBFS variant as the baselines of our development.

%%%% 2.1. Monotone operators and related concepts
\beforesubsec
\subsection{Monotone-type operators and related concepts}
\aftersubsec
In this paper, we work with a finite dimensional space $\R^p$ equipped with the standard inner product $\iprods{\cdot, \cdot}$ and Euclidean norm $\norms{\cdot}$.
For a single-valued or multivalued mapping $G : \R^p \rightrightarrows 2^{\R^p}$, $\dom{G} = \set{x \in\R^p : Gx \not= \emptyset}$ denotes its domain, $\ran{G} := \set{ u \in \R^p : u \in Gx, ~x\in\dom{G}}$ denotes its range, $\gra{G} = \set{(x, y) \in \R^p\times \R^p : y \in Gx}$ denotes its graph, where $2^{\R^p}$ is the set of all subsets of $\R^p$.
The inverse of $G$ is defined by $G^{-1}y := \sets{x \in \R^p : y \in Gx}$.

%%% a. Monotonicity.
\vspace{1ex}
\noindent(a)~\textbf{Monotonicity.}
For a multivalued mapping $G : \R^p \rightrightarrows 2^{\R^p}$, we say that $G$ is $\mu$-monotone for some $\mu \in \R$ if $\iprods{u - v, x - y} \geq \mu\norms{x - y}^2$ for all $(x, u), (y, v) \in \gra{G}$.
If $\mu = 0$, then $G$ is called monotone, i.e. $\iprods{u - v, x - y} \geq 0$ for all $(x, u), (y, v) \in \gra{G}$. 
If $\mu > 0$, then $G$ is called strongly monotone with a strong monotonicity parameter $\mu$.
If $\mu < 0$, then we say that $G$ is $\vert\mu\vert$-hypomonotone or $\vert\mu\vert$-weakly monotone.
If $G$ is single-valued, then the above condition reduces to $\iprods{Gx - Gy, x - y} \geq \mu\norms{x - y}^2$ for all $x, y\in\dom{G}$, which defines the corresponding concepts for the single-valued mapping $G$.
We say that $G$ is $\mu$-maximally monotone if $\gra{G}$ is not properly contained in the graph of any other $\mu$-monotone operator.

%%% b. Lipschitz continuity and co-coerciveness.
\vspace{1ex}
\noindent(b)~\textbf{Lipschitz continuity.}
A single-valued operator $G$ is called $L$-Lipschitz continuous if $\norms{Gx - Gy} \leq L\norms{x - y}$ for all $x, y\in\dom{G}$, where $L \geq 0$ is a Lipschitz constant. 
If $L = 1$, then we say that $G$ is nonexpansive, while if $L \in [0, 1)$, then we say that $G$ is $L$-contractive, and $L$ is its contraction factor.
\revised{If $G$ is $L$-Lipschitz continuous, then by the Cauchy-Schwarz inequality, we have $\iprods{Gx - Gy, x - y} \geq -\norms{Gx - Gy}\norms{x - y} \geq - L\norms{x - y}^2$, showing that $G$ is also $L$-hypomonotone.}

\vspace{1ex}
\noindent(c)~\textbf{Co-hypomonotonicity and co-coerciveness.}
\revised{A mapping $G$ is said to be $\rho$-co-hypomonotone for $\rho > 0$ if $\iprods{u - v, x - y} \geq -\rho\norms{u - v}^2$ for all $(x, u), (y, v) \in \gra{G}$.
The co-hypomonotonicity is also referred to as $-\rho$-co-monotonicity with a parameter $\rho$, see, e.g., \cite{bauschke2020generalized,combettes2004proximal}.
Clearly, a co-hypomonotone operator is not necessarily monotone.
If $\rho = 0$, then $G$ is just monotone, while if $\rho < 0$, then $G$ is $\beta$-co-coercive with the parameter $\beta := -\rho$.
If $\beta = 1$, then we say that $G$ is firmly nonexpansive.
If $G$ is $\beta$-co-coercive, then it is also monotone and $\frac{1}{\beta}$-Lipschitz continuous (by the Cauchy-Schwarz inequality), but the reversed statement is not true in general.

The co-hypomonotonicity concept was extended to a so-called \textit{semimonotonicity} in \cite{evens2023convergence,otero2011regularity}.
It is stronger than the weak Minty solution condition, e.g., in \cite{diakonikolas2021efficient}, which only requires to hold at a given solution $x^{\star}$, i.e. $\iprods{u, x - x^{\star}} \geq -\rho \norms{u}^2$ for all $(x, u) \in \gra{G}$.
It is obvious that $G$ is $\rho$-co-hypomonotone if and only if $G^{-1}$ is $\rho$-hypomonotone, i.e. $\iprods{u - v, x - y} \geq -\rho\norms{u-v}^2$ for all $(u, x), (v, y) \in \gra{G^{-1}}$.

The co-hymonotonicity is also related to a so-called ``interaction dominance condition'' studied in \cite{grimmer2023landscape}.
Proposition 4.6 in \cite{evens2023convergence} states that if a minimax objective function satisfies an interaction dominance condition, then its gradient is co-hypomonotone.
This structure presents in various minimax models  \cite{grimmer2023landscape}.

Another related concept  is the \textit{strong metric subregularity} (see \cite{cibulka2018strong,Rockafellar2004}) widely studied in variational analysis, though it is often defined locally.
More specifically, a mapping $G$ is said to be \textit{strongly metrically subregular} at $x$ for $u$ when $u \in Gx$ and there exists  $\kappa > 0$ along with a neighborhood $\Xc$ of $x$ and $\Uc$ of $u$ such that
\begin{equation*} 
d(y, x) \leq \kappa \hat{d}(u, Gy \cap\Uc), \quad \forall y \in \Xc,
\end{equation*}
where $d$ and $\hat{d}$ are given metrics.
Note that the [strong] metric subregularity is classical and relates to many other common concepts in variational analysis such as sharp minimizer, Aubin's property, calmness, and error bound condition. 

More concretely, if $G$ is  strongly metric subregular at any $x$ and additionally $\mu$-hypomonotone (i.e. $\iprods{u - v, x - y} \geq -\mu\norms{x - y}^2$ for all $(x, u), (y, v) \in \gra{G}$.
If the metrics $d$ and $\hat{d}$ are the standard Euclidean distances, i.e. $d(y, x) = \norms{y-x}$ and $\hat{d}(u, v) = \norms{u - v}$, then we get $\iprods{u - v, x - y} \geq -\mu\norms{x - y}^2 \geq -\mu\kappa^2\norms{u - v}^2$, showing that $G$ is also $\mu\kappa^2$-co-hypomonotone.

}

%%%% c. Resolvent and reflection operators.
\vspace{1ex}
\noindent\textbf{Resolvent operator.}
Given a possibly multivalued mapping $T$, the operator $$J_Tu := \set{x \in \R^p : u \in x + Tx}$$ is called the resolvent of $T$, often denoted by $J_Tu = (\Id + T)^{-1}u$, where $\Id$ is the identity mapping.
Clearly, evaluating $J_T$ requires solving an inclusion $0 \in y - x + Ty$ in $y$ for given $x$.
If $T$ is monotone, then $J_T$ is singled-valued, and if $T$ is maximally monotone, then $J_T$ is singled-valued and $\dom{J_T} = \R^p$.
If $T$ is monotone, then $J_T$ is also firmly nonexpansive \cite[Proposition 23.10]{Bauschke2011}, and hence nonexpansive. 
Moreover $0 \in Tx$ if and only if $x$ is a fixed-point of $J_T$.
\revised{If $T$ is $\mu$-hypomonotone, then for $\Id + \eta T$ is also $(1-\eta\mu)$-strongly monotone, provided that $\eta\mu < 1$.
In this case, $J_{\eta T}$ remains single-valued and $(1-\eta\mu)$-co-coercive.}

%%%% 2.2. Halpern fixed-point iteration and its relation to Nesterov's scheme
\beforesubsec
\subsection{Preliminary results}\label{subsec:preliminary}
\aftersubsec
Let us recall the following preliminary results, which will be used in the sequel.

\vspace{1ex}
\noindent\textbf{Solution characterization.}
Let us characterize solutions of \eqref{eq:coMI}.
Assume that for some $\eta > 0$, the resolvent $J_{\eta T} = (\Id + \eta T)^{-1}$ associated with \eqref{eq:coMI} exists and single-valued.
Let us define the following residual operator associated with \eqref{eq:coMI}:
\begin{equation}\label{eq:coMI_residual}
\mcal{G}_{\eta}x := \tfrac{1}{\eta}\left( x - J_{\eta T}(x - \eta Fx) \right).
\end{equation}
Then, $x^{\star} \in \zer{F+T}$ if and only if $\Gc_{\eta}x^{\star} = 0$.
Our goal is to approximate $x^{\star}$ by $x^k$ generated from a given algorithm after $k$ iterations such that $ \norms{\Gc_{\eta}x^k}  \leq \varepsilon$ for a given tolerance $\varepsilon > 0$.
We also characterize the convergence rate of $ \norms{\Gc_{\eta}x^k}$.

Another way of characterizing approximate solutions of \eqref{eq:coMI} is to use a residual $r^k := Fx^k + \xi^k$ of \eqref{eq:coMI} at $x^k$ for some $\xi^k \in Tx^k$.
Since $T$ is multivalued, there may exist different residuals $r^k$ at $x^k$.
Clearly, since $r^k \in Fx^k + Tx^k$, if $\norms{r^k} = 0$, then $x^k$ is an exact solution of \eqref{eq:coMI}.
Therefore, if $\norms{r^k} \leq \varepsilon$ for some tolerance $\varepsilon > 0$, then $x^k$ can be viewed as an $\varepsilon$-approximate solution of \eqref{eq:coMI}.
We call $r^k$ a residual of \eqref{eq:coMI} at $x^k$.
In this paper, we will upper bound the norm $\norms{r^k}$ of the residual $r^k$ at $x^k$ to characterize an approximate solution $x^k$.
If $x^k$ is the last iterate computed by the algorithm, then we refer to the convergence rate on $\norms{r^k}$ as the last-iterate convergence rate.
If $x^k = x^{k_b}$ with $k_b := \argmin\sets{ \norms{r^i} : 0 \leq i \leq k}$, then we call the convergence rate on $\norms{r^{k_b}}$ the best-iterate convergence rate.

\vspace{1ex}
\noindent\textbf{A brief review of extragradient-type methods.}
The classical extragradient method \cite{diakonikolas2021efficient,Facchinei2003,korpelevich1976extragradient} is often applied to solve monotone variational inequalities and their special cases.
To solve the inclusion \eqref{eq:coMI}, its modification was proposed by Tseng in \cite{tseng2000modified}, which is known as a forward-backward-forward splitting (FBFS) method.
This scheme is described as follows:
Given $y^0 \in \R^p$, we update 
\begin{equation}\label{eq:EG}
\arraycolsep=0.2em
\left\{\begin{array}{lcl}
x^k &:= & J_{ \eta T}(y^k - \eta Fy^k), \vspace{1ex}\\
y^{k+1} & := & x^k - (\hat{\eta} Fx^k - \eta Fy^k),
\end{array}\right.
\tag{FBFS}
\end{equation}
where $\eta > 0$ and $\hat{\eta} > 0$ are given parameter.
Very often, we choose $\hat{\eta} = \eta \in \left(0, \frac{1}{L}\right)$, where $L$ is the Lipschitz constant of $F$.
However, we can also choose $\hat{\eta} = \frac{\eta}{\beta}$ for some $0 < \beta \leq 1$ as in \cite{diakonikolas2021efficient,luo2022last}, which is called the extragradient plus (EG+) scheme.

\revised{Note that if we replace $Fy^k$ by $Fx^{k-1}$ in the extragradient method from \cite{korpelevich1976extragradient}, then we obtain Popov's past-extragradient method in \cite{popov1980modification}.
Similarly, if we replace $Fy^k$ in \eqref{eq:EG} by $Fx^{k-1}$, then  we obtain the following past FBFS scheme:
\begin{equation}\label{eq:PEG}
\arraycolsep=0.2em
\left\{\begin{array}{lcl}
x^k &:= & J_{ \eta T}(y^k - \eta Fx^{k-1}), \vspace{1ex}\\
y^{k+1} & := & x^k - (\hat{\eta} Fx^k - \eta Fx^{k-1}),
\end{array}\right.
\tag{PFBFS}
\end{equation}
where $x^{-1} := y^0$ is given, $\hat{\eta} := \frac{\eta}{\beta}$ for some $\beta \in (0, 1]$, and $\eta \in \left(1, \frac{1}{3L}\right]$ is a given parameter.}
This method requires only one evaluation $Fx^k$ of $F$ and one resolvent $J_{\hat{\eta}T}$.
If we eliminate $y^k$ and choose $\hat{\eta} = \eta$, then we obtain $\hat{x}^{k+1} = J_{\eta T}(\hat{x}^k - \eta( 2F\hat{x}^k - F\hat{x}^{k-1}))$ with $\hat{x}^k := x^{k-1}$, which is exactly the forward-reflected-backward splitting scheme in \cite{malitsky2020forward}.
While the asymptotic convergence and convergence rates of both \eqref{eq:EG} and \eqref{eq:PEG} under the monotonicity of $F$ are classical, the best-iterate convergence rates of these methods under the co-hypomonotonicity of $F$ can be found, e.g., in \cite{luo2022last}.
Note that \eqref{eq:PEG} is different from the reflected forward-backward splitting method $x^{k+1} := J_{\eta T}(x^k - \eta F(2x^k - x^{k-1}))$ in \cite{cevher2021reflected} (rooted from the projected reflected gradient method in \cite{malitsky2015projected}) or the golden ratio method in \cite{malitsky2019golden} for solving monotone variational inequality, a special case of \eqref{eq:coMI}.

%%%%%%%%%%%%%%%%%%%%%%%%%%%%%%%%%%%%%%%%%
%%%%% 3. Nesterov's Accelerated Extragradient Method
%%%%%%%%%%%%%%%%%%%%%%%%%%%%%%%%%%%%%%%%%
\beforesec
\section{Nesterov's Accelerated FBFS Method for Solving \eqref{eq:coMI}}\label{sec:AEG}
\aftersec
Our goal in this section is to develop an accelerated variant of \eqref{eq:EG} to solve \eqref{eq:coMI} using Nesterov's acceleration techniques \cite{Nesterov2004}.
Motivated from recent works \cite{cai2022accelerated,cai2022baccelerated,diakonikolas2020halpern,lee2021fast,tran2021halpern,yoon2021accelerated} on Halpern's fixed-point iteration with ``optimal'' convergence rates on the norm of residual of \eqref{eq:coMI}, in this section we exploit the idea from \cite{tran2022connection} to develop an alternative ``Nesterov's accelerated'' variant of the FBFS method \eqref{eq:EG} to solve \eqref{eq:coMI}.

%%%%% 3.1. Algorithm derivation.
\beforesubsec
\subsection{Algorithmic derivation}\label{subsec:EAG_derive}
\aftersubsec
Our starting point is Nesterov's acceleration interpretation of \eqref{eq:EG} to solve \eqref{eq:coME} from \cite{tran2022connection}.
Here, we extend this variant to solve the inclusion \eqref{eq:coMI}, which is presented as follows:
Starting from an initial point $x^0 \in \R^p$, set $y^0 := x^0$ and $z^0 := x^0$, we update
\begin{equation}\label{eq:AcEG_v0}
\arraycolsep=0.2em
\left\{\begin{array}{lcl}
z^{k+1} &:= & x^k - \gamma (Fx^k + \xi^k), \vspace{1ex}\\
y^{k+1} &:= & z^{k+1} + \theta_k(z^{k+1} - z^k) + \nu_k(y^k - z^{k+1}), \vspace{1ex}\\
x^{k+1} & := & y^{k+1} - \eta(Fy^{k+1} + \xi^{k+1}) + \hat{\eta}_{k+1}(Fx^k + \xi^k), 
\end{array}\right.
\end{equation}
where $\gamma$, $\eta$, $\hat{\eta}_{k+1}$, $\theta_k$ and $\nu_k$ are given parameters, which will be determined later, and $\xi^k \in Tx^k$ is an arbitrary element of $Tx^k$.
Clearly, if $T = 0$, then \eqref{eq:AcEG_v0} exactly reduces to the one in \cite{tran2022connection}.

Since $\xi^{k+1} \in Tx^{k+1}$ appears on the right-hand side of the last line of \eqref{eq:AcEG_v0}, by conceptually using the resolvent $J_{\eta T} = (\Id + \eta T)^{-1}$ of $\eta T$, we can transform this scheme into an equivalent form that is more convenient for implementation.
First, we switch the first and last lines of \eqref{eq:AcEG_v0} and shift the index from $k+1$ to $k$ as 
\begin{equation*}
x^k := y^k - \eta(Fy^k + \xi^k) + \hat{\eta}_k(Fx^{k-1} + \xi^{k-1}).
\end{equation*}
Next, for simplicity of presentation, we introduce $w^k := Fx^k + \xi^k$.
Then, from the last expression of $x^k$, we have $ y^k - \eta Fy^k + \hat{\eta}_k w^{k-1} = x^k + \eta \xi^k \in (\Id + \eta T)x^k$.
Hence, we obtain $x^k = J_{\eta T}\big( y^k - \eta Fy^k + \hat{\eta}_k w^{k-1} \big)$ provided that $J_{\eta T}$ is single-valued.
In addition, we also have $x^k = y^k - \eta(Fy^k - Fx^k + w^k) + \hat{\eta}_kw^{k-1}$, leading to $w^k = \frac{1}{\eta}(y^k - x^k + \hat{\eta}_kw^{k-1}) + Fx^k - Fy^k$.
Finally, putting these derivations together, we obtain the following equivalent interpretation of \eqref{eq:AcEG_v0}:
\begin{equation}\label{eq:AcEG}
\arraycolsep=0.2em
\left\{\begin{array}{lcl}
\revised{ x^k } & \revised{ :=  } & \revised{ J_{\eta T}\big(y^k - \eta Fy^k + \hat{\eta}_k w^{k-1} \big), } \vspace{1ex}\\
w^k & := &   \frac{1}{\eta}(y^k - x^k + \hat{\eta}_kw^{k-1}) + Fx^k - Fy^k, \vspace{1ex}\\
z^{k+1} &:= & x^k - \gamma w^k, \vspace{1ex}\\
y^{k+1} &:= & z^{k+1} + \theta_k(z^{k+1} - z^k)  +  \nu_k(y^k - z^{k+1}), 
\end{array}\right.
\tag{AEG}
\end{equation}
where we choose $w^{-1} := \frac{\eta}{\hat{\eta}_0}w^0 = \frac{\eta}{\hat{\eta}_0}(Fx^0 + \xi^0)$ for $\xi^0 \in Tx^0$, and set $z^0 = y^0 := x^0$ for a given initial point $x^0 \in \R^p$.
%Here, we do not assume that $J_{\eta T}$ is single-valued.
%Note that when $T$ is $\mu$-monotone, then for any $\eta$ such that $\eta + \mu > 0$, $J_{\eta T}$ is single-valued and firmly nonexpansive. 
We call this scheme an \textit{accelerated FBFS method} to distinguish it from the extra-anchored gradient method in \cite{yoon2021accelerated} (see also Section~\ref{apdx:Halpern_AcEG}).

If $J_{\eta T}$ is computable, then \eqref{eq:AcEG} is implementable. %, even though $J_{\eta T}$ is not necessarily single-valued.
Nevertheless, we have not yet obviously seen  the connection between \eqref{eq:AcEG} and \eqref{eq:EG}.
Now, let us eliminate $z^k$ to obtain
\begin{equation}\label{eq:AcEG_v3}
\arraycolsep=0.2em
\left\{\begin{array}{lcl}
\revised{ x^k } & \revised{ := } & \revised{ J_{\eta T}\big(y^k - \eta Fy^k { \ + \ \hat{\eta}_k w^{k-1}} \big), } \vspace{1ex}\\
%y^{k+1} &:= & x^k - \beta_k(Fx^k - Fy^k)  { \ + \ \theta_k(x^k - x^{k-1}) } { \ + \ \hat{\beta}_k(y^k - x^k)  + \tilde{\beta}_kw^{k-1}}, \vspace{1ex}\\
w^k & := &   \frac{1}{\eta}(y^k - x^k + \hat{\eta}_kw^{k-1}) + Fx^k - Fy^k, \vspace{1ex}\\
\revised{y^{k+1}} & := & \revised{x^k + \theta_k(x^k - x^{k-1}) - \beta_k w^k + \sigma_k w^{k-1} - \eta\nu_k(Fx^k - Fy^k),}
\end{array}\right.
\end{equation}
where \revised{$\beta_k := \gamma(1 + \theta_k - \nu_k) - \eta\nu_k$ and $\sigma_k := \gamma\theta_k - \nu_k\hat{\eta}_k$.}
%where $\beta_k :=  \gamma(1 + \theta_k - \nu_k)$, $\hat{\beta}_k := \nu_k - \frac{\beta_k}{\eta}$,  and $\tilde{\beta}_k := \gamma \theta_k - \frac{\beta_k\hat{\eta}_k}{\eta}$.
We can view \eqref{eq:AcEG_v3} as an accelerated variant of \eqref{eq:EG}, where a momentum term ${\theta_k(x^k - x^{k-1})}$ and other simple correction terms are added to \eqref{eq:EG} as observed in \cite{mainge2021accelerated}.

Overall, at each iteration $k$, \eqref{eq:AcEG} requires one evaluation of the resolvent $J_{\eta T}$ of $\eta T$ and two evaluations $Fx^k$ and $Fy^k$ of $F$.
This per-iteration complexity is essentially the same as \eqref{eq:EG}.
When $T = 0$, as indicated in \cite{tran2022connection}, \eqref{eq:AcEG_v0} can be derived from the Halpern-type schemes developed in \cite{lee2021fast,yoon2021accelerated}.
The convergence rates of these Halpern-type schemes were proven in \cite{lee2021fast,yoon2021accelerated} for both the monotone and co-hypomonotone cases when $T = 0$.
When $T\neq 0$, \cite{cai2022accelerated,cai2022baccelerated} recently proposed corresponding Halpern-type schemes and analyzed their convergence rates for both the monotone and co-hypomonotone cases.
%In Section~\ref{apdx:Halpern_AcEG}, we provide an alternative convergence rate analysis of the Halpern-type methods in  \cite{cai2022accelerated,cai2022baccelerated}, which can be viewed as straightforward extensions of the results in \cite{lee2021fast,yoon2021accelerated} to handle the co-hypomonotone inclusion \eqref{eq:coMI}.

%%% Remark 1.
\revised{
\begin{remark}\label{re:general_resolvent}
We can replace the standard resolvent $J_{\eta T} := (\Id + \eta T)^{-1}$ by its generalization $J_{\eta T} := (D + \eta T)^{-1}\circ D$, where $D$ is a symmetric and positive semidefinite operator.
In addition, our methods and their convergence results on $\norms{Fx^k + v^k}$ remain valid if  $J_{\eta T}$ is a multivalued mapping and satisfies $\ran{J_{\eta T}} \subseteq \dom{F} = \R^p$ and $\dom{J_{\eta T}} = \R^p$.
These conditions guarantee the well-definedness of \eqref{eq:AcEG}.
\end{remark}
}

\beforesubsec
\subsection{Convergence analysis}\label{subsec:EAG_convergence}
\aftersubsec
To establish the convergence of \eqref{eq:AcEG}, we use the following potential or Lyapunov function from \cite{tran2022connection}:
\begin{equation}\label{eq:AcEG_potential_func}
\Pc_k := a_k\norms{w^{k-1}}^2 + b_k\iprods{w^{k-1}, z^k - y^k} + \norms{z^k + t_k(y^k - z^k) - x^{\star}}^2,
\end{equation}
where $a_k > 0$, $b_k > 0$, and $t_k > 0$ are given parameters.
Note that \cite{tran2022connection} did not provide convergence analysis for \eqref{eq:AcEG} in the co-hypomonotone case even when $T = 0$.
It only established the convergence rate of a Nesterov's accelerated variant derived from the extra-anchored gradient method in \cite{yoon2021accelerated}.
Here, we reuse this potential function to analyze the convergence of \eqref{eq:AcEG} for solving \eqref{eq:coMI} in the co-hypomonotone setting, which obviously covers the monotone case.

Let us first lower bound the potential function $\Pc_k$ as in the following lemma.

%%% Lemma 3.1.
\begin{lemma}\label{le:P_lowerbound}
Let $\sets{(x^k, y^k, z^k, w^k)}$ be generated by \eqref{eq:AcEG} and  $\Pc_k$ be defined by \eqref{eq:AcEG_potential_func}.
Then, $w^k \in Fx^k + Tx^k$.
If $J_{\eta T}$ is single-valued and nonexpansive for any $\eta > 0$, then for $\Gc_{\eta}$ defined by \eqref{eq:coMI_residual}, we have
\begin{equation}\label{eq:G_vs_w}
\norms{\Gc_{\eta}x^k} \leq \norms{w^k}.
\end{equation}
If $F+T$ is $\rho$-co-hypomonotone, then
\begin{equation}\label{eq:P_lowerbound}
\Pc_k \geq  \norms{z^k + t_k(y^k - z^k) - \tfrac{b_k}{2t_k}w^{k-1} -  x^{\star} }^2  +  \left( a_k - \tfrac{b_k^2}{4t_k^2} - \tfrac{(\gamma + \rho) b_k}{t_k}  \right) \norms{w^{k-1}}^2.
\end{equation}
Consequently, if $a_k - \frac{b_k^2}{4t_k^2} - \frac{(\gamma + \rho) b_k}{t_k} \geq 0$ and $J_{\eta T}$ is single-valued and nonexpansive, then we obtain
\begin{equation}\label{eq:P_lowerbound2}
\Pc_k \geq   \left( a_k - \frac{b_k^2}{4t_k^2} - \frac{(\gamma + \rho) b_k}{t_k}  \right) \norms{\Gc_{\eta}x^{k-1}}^2.
\end{equation}
\end{lemma}

%%% Proof of Lemma 3.1.
\begin{proof}
From the first line of \eqref{eq:AcEG}, we have $y^k - x^k + \hat{\eta}_k w^{k-1} - \eta Fy^k  \in \eta Tx^k$.
Rearranging this expression, inserting $Fx^k$ to both sides, and using $w^k$ from the second line of \eqref{eq:AcEG}, we have $w^k = \frac{1}{\eta}(y^k - x^k + \hat{\eta}_kw^{k-1}) - (Fy^k - Fx^k)  \in Fx^k + Tx^k$.
This proves that $w^k \in Fx^k + Tx^k$.

From $w^k \in Fx^k + Tx^k$, we get $x^k - \eta Fx^k + \eta w^k \in x^k + \eta Tx^k$, leading to $x^k = J_{\eta T}(x^k - \eta Fx^k + \eta w^k)$ due to the single-valued assumption of $J_{\eta T}$.
Using this relation, the definition of $\Gc_{\eta}$ from \eqref{eq:coMI_residual}, and the nonexpansiveness of $J_{\eta T}$ from our assumption, we have
\begin{equation*}
\eta\norms{\Gc_{\eta}x^k} = \norms{x^k - J_{\eta T}(x^k - \eta Fx^k)} = \norms{J_{\eta T}(x^k - \eta Fx^k + \eta w^k) - J_{\eta T}(x^k - \eta Fx^k)} \leq \eta\norms{w^k}. 
\end{equation*}
This clearly implies that $\norms{G_{\eta}x^k} \leq \norms{w^k}$ as in \eqref{eq:G_vs_w}.

Next, utilizing $z^k = x^{k-1} - \gamma w^{k-1}$ from the third line of \eqref{eq:AcEG}, and the $\rho$-co-hypomonotonicity of $F+T$ with $\iprods{w^{k-1}, x^{k-1} - x^{\star}} \geq -\rho\norms{w^{k-1}}^2$, we can show that
\begin{equation*} 
\arraycolsep=0.2em
\begin{array}{lcl}
\Pc_k  & =   & \norms{z^k + t_k(y^k - z^k) - x^{\star} - \frac{b_k}{2t_k}w^{k-1}}^2 + \left( a_k - \frac{b_k^2}{4t_k^2} -  \frac{\gamma b_k}{t_k} \right) \norms{w^{k-1}}^2 \vspace{1ex}\\
&& + {~} \frac{b_k}{t_k}\iprods{w^{k-1}, x^{k-1} - x^{\star}} \vspace{1ex}\\  
& \geq & \norms{z^k + t_k(y^k - z^k) - \frac{b_k}{2t_k}w^{k-1} -  x^{\star} }^2  +  \left( a_k - \frac{b_k^2}{4t_k^2} - \frac{(\gamma + \rho) b_k}{t_k}  \right) \norms{w^{k-1}}^2.
\end{array}
\end{equation*}
This is exactly \eqref{eq:P_lowerbound}.
Combining \eqref{eq:P_lowerbound} and $\norms{G_{\eta}x^k} \leq \norms{w^k}$, we obtain \eqref{eq:P_lowerbound2}.
\Eproof
\end{proof}
%%% End of the proof.

Our first main result is the following theorem showing the convergence rate of the \textit{accelerated FBFS} scheme \eqref{eq:AcEG}.

%%%% Theorem 3.1.
\begin{theorem}\label{th:AcEG_convergence}
Assume that $F + T$ in \eqref{eq:coMI} is $\rho$-co-hypomonotone, $F$ is $L$-Lipschitz continuous such that $2L\rho <  1$, and $x^{\star} \in \zer{F+T}$ is an arbitrary solution of \eqref{eq:coMI}.
Let $\gamma > 0$ be such that $L(2\rho + \gamma) \leq 1$ $($e.g., $\gamma := \frac{1}{L} - 2\rho > 0$$)$ and $\sets{(x^k, y^k, z^k)}$ be generated by \eqref{eq:AcEG} using the following update rules:
\begin{equation}\label{eq:AcEG_para_choice}
t_k := k + 2, \quad \eta := \gamma + 2\rho,  \quad \hat{\eta}_k = \frac{(t_k-1)\eta}{t_k},  \quad  \theta_k := \frac{t_k-1}{t_{k+1}}, \quad \text{and} \quad \nu_k := \frac{t_k}{t_{k+1}}.
\end{equation}
Then, the following bound holds:
\begin{equation}\label{eq:AcEG_convergence}
\norms{Fx^k + \xi^k}^2 \leq \frac{4\norms{x^0 - x^{\star}}^2 + 8\gamma(3\gamma+2\rho)\norms{Fx^0 + \xi^0}^2}{ \gamma^2 (k+2)^2}, \quad \text{where}\quad \xi^k \in Tx^k.
\end{equation}
If $J_{\eta T}$ is nonexpansive, then for $\Gc_{\eta}$ defined by \eqref{eq:coMI_residual}, we also have
\begin{equation}\label{eq:AcEG_convergence_b}
\norms{\Gc_{\eta}x^k}^2 \leq \frac{4\norms{x^0 - x^{\star}}^2 + 8\gamma(3\gamma+2\rho)\norms{Fx^0 + \xi^0}^2 }{ \gamma^2 (k+2)^2}.
\end{equation}
\end{theorem}

Theorem~\ref{th:AcEG_convergence} provides an alternative view of the ``accelerated'' FBFS method based on Nesterov's accelerated interpretation \cite{tran2022connection}.
Our condition $2L\rho < 1$ is the same as in \cite{cai2022accelerated,lee2021fast} for the Halpern-type variants of \eqref{eq:EG}.
Note that the estimate \eqref{eq:AcEG_convergence} only requires the $\rho$-co-hypomonotonicity of $F + T$ and the $L$-Lipschitz continuity of $F$, while \eqref{eq:AcEG_convergence_b} additionally requires the nonexpansiveness of $J_{\eta T}$.
If $\rho = 0$, i.e. $F+T$ is monotone, then our condition on the stepsize $\gamma$ becomes $L\gamma < 1$, which is used in the clasical extragradient and FBFS methods, e.g., in \cite{korpelevich1976extragradient}.
If $\rho < 0$, i.e. $F+T$ is $-\rho$-co-coercive, then our bounds remain valid as long as $-1 < L(\gamma + 2\rho) \leq 1$ (or $\max\set{-\frac{1}{L} - 2\rho, 0} < \gamma \leq \frac{1}{L} - 2\rho)$.
However, we may expect a linear convergence rate instead of a sublinear rate as in Theorem~\ref{th:AcEG_convergence}.
This requires to adapt our analysis by using the techniques, e.g., in \cite{partkryu2022}, which is beyond the scope of this paper.

Note that we can choose different $\eta$ such that $L\eta \leq 1$, but not necessary to be $\eta = \gamma + 2\rho$ as in \eqref{eq:AcEG_para_choice}.
The update rule \eqref{eq:AcEG_para_choice} in Theorem~\ref{th:AcEG_convergence} only provides one choice of $\eta$ and it makes our analysis simpler.
However, other choices of $\eta$ and $\gamma$ can also work as long as they satisfy the conditions in our analysis below.

%%%% The proof of Lemma 3.1.
\begin{proof}[The proof of Theorem~\ref{th:AcEG_convergence}]
First, by inserting $(t_k-1)z^{k+1} - (t_k-1)z^{k+1}$, we can easily expand $\norms{z^k + t_k(y^k - z^k) - x^{\star}}^2$ as follows:
\begin{equation*} 
\arraycolsep=0.2em
\begin{array}{lcl}
\norms{z^k + t_k(y^k - z^k) - x^{\star}}^2 &= & \norms{z^{k+1} - x^{\star} + (t_k - 1)(z^{k+1} - z^k) + t_k(y^k - z^{k+1})}^2 \vspace{1ex}\\
&= & \norms{z^{k+1} - x^{\star}}^2 + (t_k - 1)^2\norms{z^{k+1} - z^k}^2 + t_k^2\norms{y^k - z^{k+1}}^2 \vspace{1ex}\\
&& + {~} 2(t_k -1)\iprods{z^{k+1} - z^k, z^{k+1} - x^{\star}} \vspace{1ex}\\
&& + {~} 2t_k\iprods{y^k - z^{k+1}, z^{k+1} - x^{\star}} \vspace{1ex}\\
&& + {~} 2(t_k-1)t_k\iprods{y^k - z^{k+1}, z^{k+1} - z^k}.
\end{array} 
\end{equation*}
Next, from the fourth line of \eqref{eq:AcEG}, we have $y^{k+1} - z^{k+1} = \theta_k(z^{k+1} - z^k) + \nu_k(y^k - z^{k+1})$.
Using this line, we can show that
\begin{equation*} 
\arraycolsep=0.2em
\begin{array}{lcl}
\Tc_{[1]} &:= & \norms{z^{k+1} + t_{k+1}(y^{k+1} - z^{k+1}) -  x^{\star}}^2  \vspace{1ex}\\
&= & \norms{z^{k+1} - x^{\star} + t_{k+1}[\theta_k(z^{k+1} - z^k) + \nu_k(y^k - z^{k+1})]}^2 \vspace{1ex}\\
&= & \norms{z^{k+1} - x^{\star}}^2 + t_{k+1}^2\theta_k^2\norms{z^{k+1} - z^k}^2 +  t_{k+1}^2\nu_k^2\norms{y^k  - z^{k+1} }^2 \vspace{1ex}\\
&& + {~} 2t_{k+1}\theta_k\iprods{z^{k+1} - z^k, z^{k+1} - x^{\star} } + 2t_{k+1}\nu_k\iprods{y^k - z^{k+1}, z^{k+1} - x^{\star}} \vspace{1ex}\\
&& + {~} 2t_{k+1}^2\nu_k\theta_k\iprods{y^k - z^{k+1}, z^{k+1} - z^k}.
\end{array} 
\end{equation*}
Combining the last two expressions, we can easily get
\begin{equation*} 
\arraycolsep=0.2em
\begin{array}{lcl}
\Tc_{[2]}  &:= &   \norms{z^k + t_k(y^k - z^k) - x^{\star}}^2 - \norms{z^{k+1} + t_{k+1}(y^{k+1} - z^{k+1}) -  x^{\star}}^2 \vspace{1ex}\\
&= & \left[ (t_k-1)^2 - t_{k+1}^2\theta_k^2 \right] \norms{z^{k+1} - z^k }^2 + (t_k^2  - \nu_k^2t_{k+1}^2)\norms{y^k - z^{k+1} }^2 \vspace{1ex}\\
&& + {~} 2(t_k - 1 - t_{k+1}\theta_k)\iprods{z^{k+1} - z^k, z^{k+1} - x^{\star}} \vspace{1ex}\\
&& + {~} 2(t_k - t_{k+1}\nu_k)\iprods{y^k - z^{k+1}, z^{k+1} - x^{\star}} \vspace{1ex}\\
&& + {~} 2 \left[ t_k(t_k-1) - \theta_k\nu_kt_{k+1}^2 \right] \iprods{ y^k - z^{k+1}, z^{k+1} - z^k }.
\end{array} 
\end{equation*}
Using the above expression $\Tc_{[2]}$, the definition of $\Pc_k$ from \eqref{eq:AcEG_potential_func}, and the fourth line of  \eqref{eq:AcEG} as $z^{k+1} - y^{k+1} = -\theta_k(z^{k+1} - z^k) - \nu_k(y^k - z^{k+1})$,  we can further derive
\begin{equation}\label{eq:AcEG_proof3} 
\hspace{-1ex}
\arraycolsep=0.2em
\begin{array}{lcl}
\Pc_k - \Pc_{k+1} &= &  a_k\norms{w^{k-1}}^2 - a_{k+1}\norms{w^k}^2 \vspace{1ex}\\
&& + {~} \left[ (t_k-1)^2 - t_{k+1}^2\theta_k^2 \right] \norms{z^{k+1} - z^k }^2 + (t_k^2  - \nu_k^2t_{k+1}^2)\norms{y^k - z^{k+1} }^2 \vspace{1ex}\\
&& + {~} b_k\iprods{w^k - w^{k-1}, z^{k+1} - z^k}  + b_{k+1}\iprods{\nu_kw^k - \theta_kw^{k-1}, y^k - z^{k+1}} \vspace{1ex}\\
&& + {~} \left( b_{k+1}\theta_k - b_k\right) \big[ \iprods{w^k, z^{k+1} - z^k}  + \iprods{w^{k-1}, y^k - z^{k+1}} \big]  \vspace{1ex}\\
&& + {~} 2(t_k - 1 - t_{k+1}\theta_k)\iprods{z^{k+1} - z^k, z^{k+1} - x^{\star}} \vspace{1ex}\\
&& + {~} 2(t_k - t_{k+1}\nu_k)\iprods{y^k - z^{k+1}, z^{k+1} - x^{\star}} \vspace{1ex}\\
&& + {~} 2 \left[ t_k(t_k-1) - t_{k+1}^2\theta_k\nu_k \right] \iprods{ y^k - z^{k+1}, z^{k+1} - z^k}.
\end{array} 
\hspace{-4ex}
\end{equation}
Let us choose positive parameters $t_k$, $\nu_k$, $\theta_k$, and $b_k$ such that
\begin{equation}\label{eq:AcEG_para_cond1}
\arraycolsep=0.2em
\begin{array}{lclclcl}
t_k - t_{k+1}\nu_k & = &  0, \ \ &  & t_k(t_k-1) - \nu_k\theta_kt_{k+1}^2 & = & 0, \vspace{1ex}\\
t_k - 1  - t_{k+1} \theta_k & =&  0, \ &  \text{and} \ &  b_{k+1}\theta_k - b_k & = &  0.
\end{array}
\end{equation}
These conditions lead to the update of $\theta_k$, $\nu_k$, and $b_k$ as $\theta_k = \frac{t_k - 1}{t_{k+1}}$, $\nu_k = \frac{t_k}{t_{k+1}}$, and $b_{k+1} = \frac{b_k}{\theta_k} = \frac{b_kt_{k+1}}{t_k-1}$, respectively.
Clearly, these updates of $\nu_k$ and $\theta_k$ are exactly shown in \eqref{eq:AcEG_para_choice}.

Now, under the choice of parameters as in \eqref{eq:AcEG_para_cond1}, \eqref{eq:AcEG_proof3} reduces to
\begin{equation}\label{eq:AcEG_proof5} 
\arraycolsep=0.2em
\begin{array}{lcl}
\Pc_k - \Pc_{k+1} &= &  a_k\norms{w^{k-1}}^2 - a_{k+1}\norms{w^k}^2 + b_k\iprods{w^k - w^{k-1}, z^{k+1} - z^k} \vspace{1ex}\\
&& + {~} b_{k+1} \iprods{\nu_k w^k - \theta_kw^{k-1}, y^k - z^{k+1} }.
\end{array} 
\end{equation}
As shown in Lemma~\ref{le:P_lowerbound}, we have $w^k \in Fx^k + Tx^k$.
By the $\rho$-co-hypomonotonicity of $F + T$ and $z^{k+1} = x^k - \gamma w^k$ from the third line of \eqref{eq:AcEG}, we have $\iprods{w^k - w^{k-1}, z^{k+1} - z^k} = \iprods{w^k - w^{k-1}, x^k - x^{k-1}} - \gamma\norms{w^k - w^{k-1}}^2 \geq -(\rho + \gamma)\norms{w^k - w^{k-1}}^2$.
Therefore, we get
\begin{equation}\label{eq:AcEG_proof7} 
\arraycolsep=0.2em
\begin{array}{lcl}
\iprods{w^k - w^{k-1}, z^{k+1} - z^k} &\geq & - (\gamma + \rho)\left[ \norms{w^k}^2 + \norms{w^{k-1}}^2 - 2\iprods{w^k, w^{k-1}} \right].
\end{array}
\end{equation}
For our convenience, let us denote $\hat{w}^k := w^k + Fy^k -  Fx^k$.
Then, it is obvious that $Fy^k - Fx^k = \hat{w}^k - w^k$.
From the second line of \eqref{eq:AcEG} and $Fy^k - Fx^k = \hat{w}^k - w^k$, we have $x^k - y^k = \hat{\eta}_k w^{k-1} - \eta \hat{w}^k$.
Substituting this into the third line of \eqref{eq:AcEG}, we get $y^k - z^{k+1} = \gamma w^k + \eta \hat{w}^k -  \hat{\eta}_k w^{k-1}$.
Using this expression, we can show that
\begin{equation}\label{eq:AcEG_proof6} 
\arraycolsep=0.2em
\begin{array}{lcl}
\iprods{\nu_k w^k - \theta_kw^{k-1}, y^k - z^{k+1} } &= & \iprods{\nu_k w^k - \theta_kw^{k-1}, \gamma w^k + \eta \hat{w}^k - \hat{\eta}_kw^{k-1}} \vspace{1ex}\\
&= &  \nu_k\gamma \norms{w^k}^2 +  \theta_k\hat{\eta}_k \norms{w^{k-1}}^2  + \nu_k\eta \iprods{w^k, \hat{w}^k} \vspace{1ex}\\
&& - {~} \eta \theta_k\iprods{w^{k-1}, \hat{w}^k} - (\nu_k\hat{\eta}_k  + \gamma \theta_k) \iprods{w^{k-1}, w^k}.
\end{array}
\end{equation}
Substituting \eqref{eq:AcEG_proof7} and \eqref{eq:AcEG_proof6} into \eqref{eq:AcEG_proof5}, and utilizing $b_k = b_{k+1}\theta_k$, we can show that
\begin{equation}\label{eq:AcEG_proof9} 
\hspace{-1ex}
\arraycolsep=0.2em
\begin{array}{lcl}
\Pc_k - \Pc_{k+1} &\geq &  \left[  a_k + b_k\left(\hat{\eta}_k - \gamma - \rho \right) \right]  \norms{w^{k-1}}^2  +  \left[ b_k\left(\frac{\gamma \nu_k}{\theta_k} - \gamma - \rho\right)  - a_{k+1} \right] \norms{w^k}^2 \vspace{1ex}\\
&& + {~} \eta b_k\left(\frac{\nu_k}{\theta_k} - 1\right)\iprods{w^k, \hat{w}^k}  +  \eta b_k  \iprods{w^k - w^{k-1}, \hat{w}^k } \vspace{1ex}\\
&& - {~}  b_k \left( \frac{\hat{\eta}_k\nu_k}{\theta_k}  - \gamma - 2\rho \right)  \iprods{w^{k-1}, w^k}.
\end{array} 
\hspace{-6ex}
\end{equation}
Now, using the Lipschitz continuity of $F$ and the relations $Fy^k - Fx^k = \hat{w}^k - w^k$ and $x^k - y^k = \hat{\eta}_k w^{k-1} - \eta \hat{w}^k$ as above, we have $\norms{\hat{w}^k - w^k}^2 = \norms{Fy^k - Fx^k}^2 \leq L^2\norms{x^k - y^k}^2 = L^2\norms{\eta\hat{w}^k - \hat{\eta}_k w^{k-1}}^2$.
This inequality leads to
\begin{equation*} 
\arraycolsep=0.2em
\begin{array}{lcl}
0 & \geq & \norms{w^k}^2 + (1 - L^2\eta^2)\norms{\hat{w}^k}^2 - 2\left(1 - L^2\eta\hat{\eta}_k \right)\iprods{w^k, \hat{w}^k} \vspace{1ex}\\
&&- {~}  2L^2\eta\hat{\eta}_k \iprods{w^k - w^{k-1}, \hat{w}^k } - L^2\hat{\eta}_k^2 \norms{w^{k-1}}^2. 
\end{array}
\end{equation*}
Multiplying this inequality by $\frac{b_k}{2L^2\hat{\eta}_k}$ and adding the result to \eqref{eq:AcEG_proof9}, we get
\begin{equation}\label{eq:AcEG_proof10} 
\arraycolsep=0.2em
\begin{array}{lcl}
\Pc_k - \Pc_{k+1} &\geq &  \left[ a_k + \frac{b_k}{2} \left( \hat{\eta}_k - 2\gamma - 2\rho \right) \right] \norms{w^{k-1}}^2 + \frac{b_k(1 - L^2\eta^2)}{2L^2\hat{\eta}_k}  \norms{\hat{w}^k}^2 \vspace{1ex}\\
&& + {~} \left[ b_k\left(\frac{\gamma \nu_k}{\theta_k} + \frac{1}{2L^2\hat{\eta}_k} - \gamma - \rho\right) - a_{k+1} \right] \norms{w^k}^2  \vspace{1ex}\\
&& + {~} b_k \left( \frac{\eta\nu_k}{\theta_k}  -  \frac{1}{L^2\hat{\eta}_k}  \right) \iprods{w^k, \hat{w}^k} -   b_k( \frac{\hat{\eta}_k\nu_k}{\theta_k} - 2\rho - \gamma )  \iprods{w^{k-1}, w^k}.
\end{array} 
\end{equation}
Let us choose $\eta := \gamma + 2\rho$ and $\hat{\eta}_k := \frac{\eta\theta_k}{\nu_k} = \frac{\eta(t_k-1)}{t_k}$ as in \eqref{eq:AcEG_para_choice}.
Then, using $\nu_k = \frac{t_k}{t_{k+1}}$, $\theta_k = \frac{t_k-1}{t_{k+1}}$, $b_{k+1} = \frac{b_kt_{k+1}}{t_k-1}$, and $\eta := \gamma + 2\rho$, we can simplify \eqref{eq:AcEG_proof10}  as
\begin{equation}\label{eq:AcEG_proof10_v1} 
\arraycolsep=0.2em
\begin{array}{lcl}
\Pc_k - \Pc_{k+1} &\geq &  \left( a_k - \frac{b_k(\gamma t_k + \eta)}{2t_k} \right) \norms{w^{k-1}}^2 + \left[ \frac{b_k\left(  \gamma t_k + \gamma + \eta \right) }{2(t_k-1)}  - a_{k+1} \right] \norms{w^k}^2 \vspace{1ex}\\
&& + {~}  \frac{b_k(1 - L^2\eta^2)t_k}{2L^2\eta(t_k-1)} \norms{w^k - \hat{w}^k}^2. 
\end{array} 
\end{equation}
Next, we impose the following conditions:
\begin{equation}\label{eq:AcEG_proof10_cond} 
1 - L^2\eta^2  \geq  0, \quad  a_k - \frac{b_k(\gamma t_k + \eta)}{2t_k}  \geq  0, \quad \text{and}\quad  \frac{b_k\left(  \gamma t_k + \gamma + \eta \right) }{2(t_k-1)}  - a_{k+1}  \geq  0.
\end{equation}
Then, under the conditions of \eqref{eq:AcEG_proof10_cond}, \eqref{eq:AcEG_proof10_v1} reduces to $\Pc_k - \Pc_{k+1} \geq 0$ for all $k\geq 0$.

We need to  show that with the update rules as in \eqref{eq:AcEG_para_choice},  three conditions of \eqref{eq:AcEG_proof10_cond} hold.
Indeed, the first condition of \eqref{eq:AcEG_proof10_cond} is equivalent to $L\eta = L(\gamma + 2\rho) \leq 1$.
If $2L\rho < 1$, then we can always choose $0 < \gamma \leq \frac{1}{L} - 2\rho$ such that $1 - L^2\eta^2 \geq 0$.
This is our first condition in Theorem~\ref{eq:AcEG_convergence}.

Now, let us choose $a_k := \frac{b_k(\gamma t_k + \eta)}{2t_k}$. 
Then, the second condition of \eqref{eq:AcEG_proof10_cond} automatically holds.
The third condition of \eqref{eq:AcEG_proof10_cond} becomes  $a_{k+1} \leq \frac{ b_{k+1}( \gamma t_{k + 1} + \eta ) }{2 t_{k+1}}$ due to $b_{k+1} = \frac{b_kt_{k+1}}{t_k-1}$ and $t_k = k + 2$.
By the choice of $a_k$, this condition holds with equality.
For $t_k := k+2$, we can easily obtain $b_k := \frac{b_0(k+1)(k+2)}{2}$ for some $b_0 > 0$, and consequently, $a_k =  \frac{b_k(\gamma t_k + \eta)}{2t_k} = \frac{b_0(k+1)(\gamma k + 3\gamma + 2\rho )}{4}$.

Finally, since $\Pc_{k+1} \leq \Pc_k$, by induction we get 
\begin{equation*}
\arraycolsep=0.2em
\begin{array}{lcl}
\Pc_k & \leq & \Pc_0 = a_0\norms{w^{-1}}^2 + b_0\iprods{w^{-1}, z^0 - y^0} + \norms{z^0 + t_0(y^0 - z^0) - x^{\star}}^2 \vspace{1ex}\\
& = & 4a_0\norms{w^0}^2 + \norms{x^0 - x^{\star}}^2  =  b_0(3\gamma + 2\rho)\norms{w^0}^2 + \norms{x^0 - x^{\star}}^2,
\end{array}
\end{equation*}
due to $z^0 = y^0 = x^0$ and $w^{-1} = \frac{\eta}{\hat{\eta}_0}w^0 = 2w^0$.
Combining this expression and \eqref{eq:P_lowerbound}, and then using the expressions of $a_k$ and $b_k$, we get
\begin{equation*} 
\arraycolsep=0.2em
\begin{array}{lcl}
\Pc_0 &= & \frac{b_0(3\gamma + 2\rho)}{2}\norms{w^0}^2 + \norms{x^0 - x^{\star}}^2 \vspace{1ex}\\
&\geq & \Pc_k \geq \left( a_k - \frac{b_k^2}{4t_k^2} - \frac{(\gamma + \rho) b_k}{t_k}  \right) \norms{w^{k-1}}^2 = \frac{ b_0(4\gamma - b_0)(k+1)^2}{16}  \norms{w^{k-1}}^2.
\end{array} 
\end{equation*}
If we choose $b_0 := 2\gamma$, then the last estimate leads to $ \norms{w^k}^2 \leq \frac{4\norms{x^0 - x^{\star}}^2}{\gamma^2(k+2)^2} + \frac{8(3\gamma+2\rho)}{\gamma(k+2)^2}\norms{w^0}^2$, which is exactly \eqref{eq:AcEG_convergence} since $w^0 := Fx^0 + \xi^0$.
If $J_{\eta T}$ is nonexpansive, then by \eqref{eq:G_vs_w}, we obtain \eqref{eq:AcEG_convergence_b}.
\Eproof
\end{proof}
%%%% End of proof.

%%%%%%%%%%%%%%%%%%%%%%%%%%%%%%%%%%%%%%%%%%%%
%%% 4. The past- extra-anchored gradient method.
%%%%%%%%%%%%%%%%%%%%%%%%%%%%%%%%%%%%%%%%%%%%
\beforesec
\section{Nesterov's Accelerated Past-FBFS Method for Solving \eqref{eq:coMI}}\label{sec:APEG}
\aftersec
Alternative to Section~\ref{sec:AEG}, in this section, we develop a new Nesterov's accelerated variant of the FBFS method \eqref{eq:EG} for solving \eqref{eq:coMI} by adopting the ideas from \cite{popov1980modification,tseng2000modified}.
However, our scheme is different from \eqref{eq:PEG} since it replaces $Fx^k$ by $Fy^{k-1}$ instead of replacing $Fy^k$ by $Fx^{k-1}$.

%%%%%%%%%%%%%%%%%%%%%%%%%%%%%%%%%%%%%%
%%% 4.1. Algorithm derivation.
%%%%%%%%%%%%%%%%%%%%%%%%%%%%%%%%%%%%%%
\beforesubsec
\subsection{Algorithmic derivation}\label{subsec:APEG_derive}
\aftersubsec
\revised{Since \eqref{eq:AcEG} requires two evaluations $Fx^k$ and $Fy^k$ of $F$ at each iteration as in \eqref{eq:EG}, we reduce its per-iteration complexity by replacing $Fx^k$ by $Fy^{k-1}$ computed from the previous iteration.
Rearranging the resulting scheme and using the resolvent of $\eta T$, we obtain the following new scheme:
Given $x^0\in\R^p$, set  $z^0 = y^0 := x^0$ and $\hat{w}^{-1} := \frac{\eta}{\hat{\eta}_0}w^0 \in \frac{\eta}{\hat{\eta}_0}(Fx^0 + Tx^0)$, at each iteration $k$, we update
\begin{equation}\label{eq:AcPEG}
\arraycolsep=0.2em
\left\{\begin{array}{lcl}
x^k & := & J_{\eta T}\big(y^k - \eta Fy^k +  \hat{\eta}_k\hat{w}^{k-1} \big), \vspace{1ex}\\
\hat{w}^k & := & \frac{1}{\eta}( y^k - x^k + \hat{\eta}_k \hat{w}^{k-1}), \vspace{1ex}\\
z^{k+1} &:= & x^k - \gamma \hat{w}^k, \vspace{1ex}\\
y^{k+1} &:= & z^{k+1} + \theta_k(z^{k+1} - z^k)  +  \nu_k(y^k - z^{k+1}), 
\end{array}\right.
\tag{APEG}
\end{equation}
where $\theta_k$, $\nu_k$, $\eta$, and $\gamma$ are given parameters, which will be determined later.

Note that the idea of replacing either $Fx^k$ or $Fy^k$ by a past value was proposed by Popov in \cite{popov1980modification} for the classical extragradient method, and then extended to FBFS variants as seen from \eqref{eq:PEG}.
However, unlike \eqref{eq:PEG}, we replace $Fx^k$ by $Fy^{k-1}$, leading to \eqref{eq:AcPEG}.
Clearly, at each iteration $k$,  \eqref{eq:AcPEG} only requires one evaluation of $J_{\eta T}$ and one evaluation $Fy^k$ of $F$ compared to \eqref{eq:AcEG}.

If we eliminate $z^k$ from \eqref{eq:AcPEG}, then we obtain the following scheme:
\begin{equation}\label{eq:AcPEG_impl}
\arraycolsep=0.2em
\left\{\begin{array}{lcl}
x^k & := & J_{\eta T}\big( y^k - \eta Fy^k +  \hat{\eta}_k\hat{w}^{k-1} \big), \vspace{1ex}\\
\hat{w}^k & := & \frac{1}{\eta}( y^k - x^k + \hat{\eta}_k \hat{w}^{k-1}), \vspace{1ex}\\
y^{k+1} &:= &  x^k + \theta_k(x^k - x^{k-1}) - \beta_k \hat{w}^k + \sigma_k \hat{w}^{k-1}, 
\end{array}\right.
\end{equation}
where $\beta_k := \gamma(1 + \theta_k - \nu_k) - \eta \nu_k$ and $\sigma_k := \gamma\theta_k - \hat{\eta}_k\nu_k$.
This representation uses both $\hat{w}^k$ and $\hat{w}^{k-1}$ to update $y^{k+1}$.
However, since \eqref{eq:AcPEG_impl} uses $Fy^{k-1}$ instead of $Fx^{k-1}$, it is different from \eqref{eq:PEG} as well as the reflected-forward-backward splitting method in \cite{cevher2021reflected,malitsky2015projected}.
}

%%%%%%%%%%%%%%%%%%%%%%%%%%%%%%%%%%%%%%
%%% 4.2. Convergence analysis
%%%%%%%%%%%%%%%%%%%%%%%%%%%%%%%%%%%%%%
\beforesubsec
\subsection{Convergence analysis}\label{subsec:convergence_APEG}
\aftersubsec
To form a potential function for \eqref{eq:AcPEG}, starting from the first and second lines of \eqref{eq:AcPEG}, we have $\hat{w}^k - Fy^k =  \frac{1}{\eta}(y^k - x^k + \hat{\eta}_k\hat{w}^{k-1}) - Fy^k  \in Tx^k$.
Hence, if we define $w^k :=  \hat{w}^k + Fx^k - Fy^k$, then  $w^k  \in Fx^k + Tx^k$.
Therefore, we propose to use the following potential or Lyapunov function to analyze the convergence of \eqref{eq:AcPEG}, which is similar to $\Pc_k$ in \eqref{eq:AcEG_potential_func}:
\begin{equation}\label{eq:AcPEG_potential_func}
\arraycolsep=0.2em
\begin{array}{lcl}
\hat{\Pc}_k &:= & a_k\norms{w^{k-1}}^2 + b_k\iprods{w^{k-1}, z^k - y^k} + \norms{z^k + t_k(y^k - z^k) - x^{\star}}^2 \vspace{1ex}\\
&& + {~} c_k\norms{w^{k-1} - \hat{w}^{k-1}}^2.
\end{array}
\end{equation}
where $a_k > 0$, $b_k > 0$, $c_k > 0$, and $t_k > 0$ are given parameters, determined later.

Similar to the proof of \eqref{eq:P_lowerbound}, using $z^k = x^{k-1} - \gamma \hat{w}^{k-1}$ from the third line of \eqref{eq:AcPEG}, the $\rho$-co-hypomonotonicity of $F+T$ as $\iprods{w^{k-1}, x^{k-1} - x^{\star}} \geq -\rho\norms{w^{k-1}}^2$, and Young's inequality as $-\iprods{w^{k-1}, \hat{w}^{k-1}} = -\iprods{w^{k-1}, \hat{w}^{k-1} - w^{k-1}} - \norms{w^{k-1}}^2 \geq -\frac{3}{2}\norms{w^{k-1}}^2 - \frac{1}{2}\norms{w^{k-1} - \hat{w}^{k-1}}^2$, we can lower bound $\hat{\Pc}_k$ as follows:
\begin{equation}\label{eq:Nes_coPEAG_potential_func}
\arraycolsep=0.2em
\begin{array}{lcl}
\hat{\Pc}_k   & =   & \norms{z^k + t_k(y^k - z^k) - x^{\star} - \frac{b_k}{2t_k}w^{k-1}}^2  + c_k\norms{w^{k-1} - \hat{w}^{k-1}}^2 \vspace{1ex}\\
&& + {~} \left( a_k - \frac{b_k^2}{4t_k^2}\right) \norms{w^{k-1}}^2 + \frac{b_k}{t_k}\iprods{w^{k-1}, x^{k-1} - x^{\star}} -  \frac{\gamma b_k}{t_k}\iprods{w^{k-1}, \hat{w}^{k-1}} \vspace{1ex}\\
& \geq &  \left( a_k - \frac{b_k^2}{4t_k^2} - \frac{(3\gamma + 2\rho) b_k}{2t_k} \right) \norms{w^{k-1}}^2 + \left( c_k - \frac{\gamma b_k}{2t_k} \right) \norms{w^{k-1} - \hat{w}^{k-1}}^2.
\end{array}
\end{equation}
Now, we are ready to establish the convergence of \ref{eq:AcPEG} in the following theorem.

%%%% Theorem 4.1.
\begin{theorem}\label{th:AcPEG_convergence}
Assume that $F + T$ in \eqref{eq:coMI} is $\rho$-co-hypomonotone, $F$ is $L$-Lipschitz continuous such that $8\sqrt{3}L\rho <  1$, and $x^{\star} \in \zer{F+T}$ is an arbitrary solution of \eqref{eq:coMI}.
Let $\gamma > 0$ be such that $16L^2 \big[ 3(3\gamma + 2\rho)^2 + \gamma(2\gamma + \rho) \big]  \leq 1$, which always exists, and $\sets{(x^k, y^k, z^k)}$ be generated by \eqref{eq:AcPEG}  using
\begin{equation}\label{eq:AcPEG_para_choice}
t_k := k + 2, \ \ \eta := 2(3\gamma + 2\rho), \ \ \hat{\eta}_k = \frac{(t_k-1)\eta}{t_k},  \ \ \theta_k := \frac{t_k-1}{t_{k+1}}, \ \text{and}  \ \ \nu_k := \frac{t_k}{t_{k+1}}.
\end{equation}
Then, for $k\geq 0$, the following bound holds:
\begin{equation}\label{eq:AcPEG_convergence}
\norms{Fx^k + \xi^k}^2 \leq \frac{4 \norms{x^0 - x^{\star}}^2 + 8\gamma(5\gamma + 2\rho) \norms{Fx^0 + \xi^0}^2}{\gamma^2(k+2)(k+4)}, \quad\text{where} \quad \xi^k \in Tx^k.
\end{equation}
\revised{Consequently,  if $J_{\eta T}$ is nonexpansive, then for $\Gc_{\eta}$ defined by \eqref{eq:coMI_residual}, we have}
\begin{equation}\label{eq:AcPEG_convergence_b}
\revised{\norms{\Gc_{\eta}x^k}^2 \leq \frac{4\norms{x^0 - x^{\star}}^2 + 8\gamma(5\gamma + 2\rho) \norms{Fx^0 + \xi^0}^2}{\gamma^2(k+2)(k+4)}.}
\end{equation}
\end{theorem}

%%% Remark 2. Choice of parameters.
\begin{remark}[\textit{\textbf{Choice of $\gamma$}}]\label{re:choice_of_param1}
\revised{
If $8\sqrt{3}L\rho < 1$, then $\bar{\gamma} := \frac{\sqrt{29 - 92L^2\rho^2} - 74 L\rho}{116L} > 0$.
In this case, if we choose $\gamma \in (0, \bar{\gamma}]$, then the condition $16L^2 [ 3(3\gamma + 2\rho)^2 + \gamma(2\gamma + \rho) ] \leq 1$ in Theorem~\ref{th:AcPEG_convergence} holds.
}
\end{remark}

%%%% The proof of Lemma 3.1.
\begin{proof}[The proof of Theorem~\ref{th:AcPEG_convergence}]
Similar to the proof of \eqref{eq:AcEG_proof5}  in Theorem~\ref{th:AcEG_convergence}, using the last line of \eqref{eq:AcPEG}, the definition \eqref{eq:AcPEG_potential_func} of $\hat{\Pc}_k$, and the update rules \eqref{eq:AcPEG_para_choice},  we can show that
\begin{equation}\label{eq:AcPEG_proof5} 
\hspace{-2ex}
\arraycolsep=0.2em
\begin{array}{lcl}
\hat{\Pc}_k - \hat{\Pc}_{k+1} &= &   a_k\norms{w^{k-1}}^2 - a_{k+1}\norms{w^k}^2  + c_k\norms{w^{k-1} - \hat{w}^{k-1}}^2 - c_{k+1}\norms{w^k - \hat{w}^k}^2 \vspace{1ex}\\
&& + {~} b_k\iprods{w^k - w^{k-1}, z^{k+1} - z^k}  + b_{k+1} \iprods{\nu_k w^k - \theta_kw^{k-1}, y^k - z^{k+1}}.
\end{array} 
\hspace{-2ex}
\end{equation}
By the $\rho$-co-hypomonotonicity of $F+T$, we have $\iprods{w^k - w^{k-1}, x^k - x^{k-1}} \geq -\rho\norms{w^k - w^{k-1}}^2$.
Using this expression and $z^{k+1} = x^k - \gamma\hat{w}^k = x^k - \gamma w^k + \gamma(w^k - \hat{w}^k)$ from the third line of \eqref{eq:AcPEG} and both the Cauchy-Schwarz and Young inequalities, we can derive that
\begin{equation}\label{eq:AcPEG_proof6} 
\hspace{-0.2ex}
\arraycolsep=0.2em
\begin{array}{lcl}
\iprods{w^k - w^{k-1}, z^{k+1} - z^k}  &= & \iprods{w^k - w^{k-1}, x^k - x^{k-1}} -  \gamma \iprods{w^k - w^{k-1},  \hat{w}^k - \hat{w}^{k-1}} \vspace{1ex}\\
& \geq &  \gamma\iprods{w^k - w^{k-1}, (w^k - \hat{w}^k) - (w^{k-1} - \hat{w}^{k-1})} \vspace{1ex}\\
&& - {~} (\gamma + \rho ) \norms{w^k - w^{k-1}}^2 \vspace{1ex}\\
& \geq & -\left(2\gamma + \rho \right) \norms{w^k - w^{k-1}}^2 -  \frac{\gamma}{2}\norms{w^k - \hat{w}^k}^2 \vspace{1ex}\\
&& - {~} \frac{\gamma}{2}\norms{w^{k-1} - \hat{w}^{k-1}}^2.
\end{array}
\hspace{-4ex}
\end{equation}
Alternatively, from the second lines of \eqref{eq:AcPEG}, we have $x^k = y^k - \eta \hat{w}^k +    \hat{\eta}_k \hat{w}^{k-1}$.
Combining this expression and the third line of \eqref{eq:AcPEG}, we get $y^k - z^{k+1} = (\gamma  + \eta)\hat{w}^k  - \hat{\eta}_k \hat{w}^{k-1} = \gamma w^k + \eta \hat{w}^k - \hat{\eta}_kw^{k-1} + \gamma(\hat{w}^k - w^k) - \hat{\eta}_k(\hat{w}^{k-1} - w^{k-1})$.
Utilizing this expression, and both the Cauchy-Schwarz and Young inequalities again, for any $\beta > 0$, we can show that
\begin{equation}\label{eq:AcPEG_proof7} 
\hspace{-0.0ex}
\arraycolsep=0.2em
\begin{array}{lcl}
\Tc_{[3]} &:= & \iprods{\nu_k w^k - \theta_kw^{k-1}, y^k - z^{k+1}} \vspace{1ex}\\
&= & \iprods{\nu_k w^k - \theta_kw^{k-1}, \gamma w^k + \eta \hat{w}^k - \hat{\eta}_kw^{k-1} } \vspace{1ex}\\
&& + {~} \iprods{\nu_k w^k - \theta_kw^{k-1}, \gamma(\hat{w}^k - w^k) - \hat{\eta}_k(\hat{w}^{k-1} - w^{k-1})} \vspace{1ex}\\
& \geq & \gamma\nu_k \norms{w^k}^2 + \hat{\eta}_k\theta_k\norms{w^{k-1}}^2  - (\hat{\eta}_k\nu_k + \gamma\theta_k)\iprods{w^k, w^{k-1}} \vspace{1ex}\\
&& - {~} \eta\theta_k\iprods{\hat{w}^k, w^{k-1}} +  \eta\nu_k\iprods{w^k, \hat{w}^k} - \frac{\beta }{2\nu_k}\norms{\nu_k w^k - \theta_kw^{k-1}}^2  \vspace{1ex}\\
&& - {~} \frac{\gamma^2 \nu_k}{\beta} \norms{w^k - \hat{w}^k}^2 - \frac{\hat{\eta}_k^2\nu_k}{\beta }\norms{w^{k-1} - \hat{w}^{k-1}}^2.
\end{array}
\hspace{-2.0ex}
\end{equation}
Further expanding \eqref{eq:AcPEG_proof6} and \eqref{eq:AcPEG_proof7}, and then substituting their results  into \eqref{eq:AcPEG_proof5}, and using $b_{k+1}\theta_k = b_k$ from \eqref{eq:AcEG_para_cond1}, we can derive
\begin{equation}\label{eq:AcPEG_proof8} 
\hspace{-0ex}
\arraycolsep=0.1em
\begin{array}{lcl}
\hat{\Pc}_k - \hat{\Pc}_{k+1} &\geq &     \left[ c_k - b_k \left( \frac{\gamma}{2} + \frac{\hat{\eta}_k^2\nu_k}{\beta\theta_k} \right) \right] \norms{w^{k-1} - \hat{w}^{k-1}}^2 \vspace{1ex}\\
&& - {~} \left[ c_{k+1} + b_k \left( \frac{\gamma}{2} + \frac{\gamma^2\nu_k}{\beta\theta_k}\right)  \right] \norms{w^k - \hat{w}^k}^2 \vspace{1ex}\\
&& + {~} \left[ a_k - b_k\left(2\gamma + \rho + \frac{\beta\theta_k}{2\nu_k} - \hat{\eta}_k\right) \right] \norms{w^{k-1}}^2 \vspace{1ex}\\
&& + {~}  \left[ b_k\left( \frac{(2\gamma - \beta) \nu_k}{2\theta_k}  - 2\gamma - \rho \right) - a_{k+1} \right] \norms{w^k}^2  \vspace{1ex}\\
&& + {~} b_k\left( 3\gamma + 2\rho + \beta -  \frac{\nu_k\hat{\eta}_k}{\theta_k}  \right) \iprods{w^k, w^{k-1}} + \eta b_k \iprods{\hat{w}^k, w^k - w^{k-1}} \vspace{1ex}\\
&& + {~}   \eta b_k \left(\frac{\nu_k}{\theta_k} - 1\right) \iprods{w^k, \hat{w}^k}.
\end{array} 
\hspace{-0ex}
\end{equation}
Next, by the Lipschitz continuity of $F$, $w^k - \hat{w}^k = Fx^k - Fy^k$, and $y^k - x^k =   \eta \hat{w}^k -   \hat{\eta}_k \hat{w}^{k-1}$, we can show that $\norms{w^k - \hat{w}^k}^2 = \norms{Fx^k - Fy^k}^2 \leq L^2\norms{x^k - y^k}^2 = L^2\norms{\eta \hat{w}^k -  \hat{\eta}_k \hat{w}^{k-1}}^2$.
Therefore, for any $\omega > 0$, using Young's inequality, this expression leads to $0 \geq \omega\norms{w^k - \hat{w}^k}^2 + \norms{w^k - \hat{w}^k}^2 - 2(1+\omega)L^2 \norms{ \eta \hat{w}^k - \hat{\eta}_kw^{k-1} }^2 - 2(1+\omega)L^2\hat{\eta}_k^2\norms{w^{k-1} - \hat{w}^{k-1} }^2$.
Denoting $M := 2(1+\omega)L^2$ and expanding the last inequality, we get
\begin{equation*} 
\arraycolsep=0.2em
\begin{array}{lcl}
0 & \geq & \omega\norms{w^k - \hat{w}^k}^2 + \norms{w^k}^2 + ( 1 - M \eta^2) \norms{\hat{w}^k}^2 - 2(1 - M\eta\hat{\eta}_k) \iprods{w^k, \hat{w}^k}  \vspace{1ex}\\
&&- {~}   2M\eta \hat{\eta}_k \iprods{\hat{w}^k, w^k - w^{k-1}} -  M\hat{\eta}_k^2 \norms{ w^{k-1}}^2 - M\hat{\eta}_k^2\norms{w^{k-1} - \hat{w}^{k-1}}^2.
\end{array}
\end{equation*}
Multiplying this inequality by $\frac{b_k}{2M\hat{\eta}_k}$ and adding the result to \eqref{eq:AcPEG_proof8}, we obtain
\begin{equation*} 
\arraycolsep=0.2em
\begin{array}{lcl}
\hat{\Pc}_k - \hat{\Pc}_{k+1} &\geq &  \left[ c_k - b_k \left( \frac{\gamma}{2} + \frac{\hat{\eta}_k^2\nu_k}{\beta\theta_k} + \frac{\hat{\eta}_k}{2}\right)  \right] \norms{w^{k-1} - \hat{w}^{k-1}}^2 \vspace{1ex}\\
&& + {~}  \left[ b_k  \left(\frac{\omega}{2M\hat{\eta}_k} - \frac{\gamma}{2} - \frac{\gamma^2 \nu_k}{\beta \theta_k}\right) -  c_{k+1} \right] \norms{w^k - \hat{w}^k}^2 \vspace{1ex}\\
&& + {~}  \left[ a_k - b_k\left(2\gamma + \rho + \frac{\beta\theta_k}{2\nu_k} - \frac{\hat{\eta}_k}{2}\right) \right] \norms{w^{k-1}}^2 + \frac{(1-M\eta^2)b_k}{2M\hat{\eta}_k}\norms{\hat{w}^k}^2 \vspace{1ex}\\
&& + {~} \left[ b_k\left( \frac{1}{2M\hat{\eta}_k} + \frac{(2\gamma - \beta) \nu_k}{2\theta_k}  - 2\gamma - \rho \right) - a_{k+1} \right] \norms{w^k}^2  \vspace{1ex}\\
&& + {~} b_k\left( 3\gamma +  2\rho  + \beta - \frac{\nu_k\hat{\eta}_k}{\theta_k}  \right) \iprods{w^k, w^{k-1}}  - b_k\left( \frac{1}{M\hat{\eta}_k} -  \frac{\eta \nu_k}{\theta_k} \right) \iprods{w^k, \hat{w}^k}.
\end{array} 
\end{equation*}
Let us choose $\beta := 3\gamma + 2\rho$, and  $\eta := 2( 3\gamma + 2\rho) = 2\beta$ and  $\hat{\eta}_k := \frac{\eta \theta_k}{\nu_k}$.
Then, we have $3\gamma + 2\rho + \beta  - \frac{\nu_k\hat{\eta}_k}{\theta_k} = 0$.
In this case, we can simplify the last expression as
\begin{equation}\label{eq:AcPEG_proof8_c} 
\arraycolsep=0.2em
\begin{array}{lcl}
\hat{\Pc}_k - \hat{\Pc}_{k+1} &\geq &  \left[ c_k - \frac{b_k}{2}\big( \gamma +  \frac{5\eta\theta_k}{\nu_k} \big) \right] \norms{w^{k-1} - \hat{w}^{k-1}}^2  \vspace{1ex}\\
&& + {~}  \left[ \frac{b_k}{2\theta_k} \big( \frac{ (\omega - 4M\gamma^2) \nu_k}{M\eta} - \gamma\theta_k  \big) - c_{k+1} \right] \norms{w^k - \hat{w}^k}^2 \vspace{1ex}\\
&& + {~}  \left[ a_k - b_k \big(  2\gamma + \rho  -  \frac{\eta\theta_k}{4\nu_k} \big)   \right] \norms{w^{k-1}}^2  \vspace{1ex}\\
&& + {~} \frac{(1 - M\eta^2)b_k\nu_k}{2M\eta\theta_k}\left( \norms{\hat{w}^k}^2 - 2\iprods{w^k, \hat{w}^k} \right) \vspace{1ex}\\
&& + {~} \left[ b_k\big(  \frac{\nu_k}{2M\eta\theta_k} + \frac{(4\gamma - \eta)\nu_k}{4\theta_k} - 2\gamma - \rho \big) - a_{k+1} \right] \norms{w^k}^2 \vspace{1ex}\\
& = &   \left[ c_k - \frac{b_k}{2}\big( \gamma +  \frac{5\eta\theta_k}{\nu_k} \big) \right] \norms{w^{k-1} - \hat{w}^{k-1}}^2  \vspace{1ex}\\
&& + {~}  \left[ \frac{b_k}{2\theta_k} \big( \frac{(1+\omega - M\eta^2 - 4M\gamma^2) \nu_k}{M\eta} - \gamma\theta_k  \big) - c_{k+1} \right] \norms{w^k - \hat{w}^k}^2 \vspace{1ex}\\
&& + {~}  \left[ a_k - b_k \big(  2\gamma + \rho  -  \frac{\eta\theta_k}{4\nu_k} \big)   \right] \norms{w^{k-1}}^2  \vspace{1ex}\\
&& + {~} \left[ b_k\big(  \frac{(\eta + 4\gamma) \nu_k}{4\theta_k} - 2\gamma - \rho \big) - a_{k+1} \right] \norms{w^k}^2.
\end{array} 
\end{equation}
Let us impose the following conditions on the parameters:
\begin{equation}\label{eq:AcPEG_proof9_para_cond2}  
\arraycolsep=0.2em
\left\{ \begin{array}{lcl}
1 + \omega - M\eta^2 &\geq & 0, \vspace{1ex}\\
a_k -   b_k \big(  2\gamma + \rho   -  \frac{\eta\theta_k}{4\nu_k} \big) & \geq & 0,  \vspace{1ex}\\
b_k\left[ \frac{(\eta + 4\gamma) \nu_k}{4\theta_k} - 2\gamma - \rho \right] - a_{k+1} & \geq & 0, \vspace{1ex}\\
c_k - \frac{b_k}{2} \big( \gamma + \frac{5\eta\theta_k}{\nu_k} \big) \ & \geq & 0,  \vspace{1ex}\\
\frac{b_k}{2\theta_k} \left[  \frac{(1 + \omega - M\eta^2 - 4M\gamma^2) \nu_k}{M\eta}  - \gamma\theta_k  \right]  - c_{k+1}  & \geq & 0.
\end{array}\right. 
\end{equation}
Then, \eqref{eq:AcPEG_proof8_c} reduces to $\hat{\Pc}_{k+1} \leq \hat{\Pc}_k$ for all $k\geq 0$.

We still update $b_k = \frac{b_0(k+1)(k+2)}{2}$ as in \eqref{eq:AcEG}, which leads to $b_{k+1} = \frac{b_k}{\theta_k} = \frac{b_kt_{k+1}}{t_k-1}$ by  \eqref{eq:AcPEG_para_choice}.
Our next step is to show that the parameters updated by  \eqref{eq:AcPEG_para_choice} guarantee the five conditions of \eqref{eq:AcPEG_proof9_para_cond2}.
First, let us choose $a_k :=  b_k \left(  2\gamma + \rho   -  \frac{\eta \theta_k}{4\nu_k}\right) = \frac{b_k( 2\gamma t_k + \eta) }{4t_k} = \frac{b_0(k+1)( \gamma k + 5\gamma + 2\rho)}{4}$.
Then, the second condition of \eqref{eq:AcPEG_proof9_para_cond2} is automatically satisfied. 
Since $b_{k+1} = \frac{b_kt_{k+1}}{t_k - 1}$ and $t_{k+1} = t_k + 1$ by \eqref{eq:AcPEG_para_choice}, the choice of $a_{k+1} = \frac{b_0(k+2)(\gamma k + 6\gamma + 2\rho)}{4}$ also guarantees the third condition of \eqref{eq:AcPEG_proof9_para_cond2}.

If we choose $c_k := \frac{b_k}{2} \big( \gamma + \frac{5\eta\theta_k}{\nu_k} \big) = \frac{b_0(k+1)[ (31\gamma + 20\rho)(k+1) + \gamma ] }{4}$, then the fourth condition of \eqref{eq:AcPEG_proof9_para_cond2} is obviously satisfied with an equality.
The last condition of \eqref{eq:AcPEG_proof9_para_cond2} holds if $M(6\eta^2 + 2\gamma\eta + 4\gamma^2) \leq 1 + \omega$.
Using $\eta = 2(3\gamma + 2\rho)$ and $M = 2(1+\omega)L^2$ for any $\omega > 0$, the last condition is equivalent to 
\begin{equation*}
8L^2[ 6(3\gamma + 2\rho)^2 + \gamma(3\gamma + 2\rho) + \gamma^2] \leq 1.
\end{equation*}
This condition is exactly equivalent to $16L^2 [ 3(3\gamma + 2\rho)^2 + \gamma(2\gamma + \rho) ] \leq 1$ in Theorem~\ref{th:AcPEG_convergence}.
Moreover, it also guarantee that $2L^2\eta^2 \leq 1$, which shows that the first condition of \eqref{eq:AcPEG_proof9_para_cond2} holds.

Using the parameters in \eqref{eq:AcPEG_para_choice}, $a_k$, $b_k$, and $c_k$ above, and \eqref{eq:Nes_coPEAG_potential_func}, we can show that
\begin{equation*} 
\arraycolsep=0.2em
\begin{array}{lcl}
\hat{\Pc}_k &\geq &  \left( a_k - \frac{b_k^2}{4t_k^2} - \frac{(3\gamma + 2\rho) b_k}{2t_k} \right) \norms{w^{k-1}}^2 + \left( c_k - \frac{\gamma b_k}{2t_k} \right) \norms{w^{k-1} - \hat{w}^{k-1}}^2 \vspace{1ex}\\
&= & \frac{b_0(k+1)[ (4\gamma - b_0) (k+1) + 4\gamma]}{16}  \norms{w^{k-1}}^2 +  \frac{(31\gamma + 20\rho) b_0(k+1)^2 }{4} \norms{w^{k-1} - \hat{w}^{k-1}}^2.
\end{array} 
\end{equation*}
%$a_k - \frac{b_k^2}{4t_k^2} - \frac{(3\gamma + 2\rho) b_k}{2t_k} =  \frac{b_0(k+1)( \gamma k + 5\gamma + 2\rho)}{4} - \frac{b_0^2(k+1)^2}{16} - \frac{(3\gamma + 2\rho) b_0(k+1)}{4} = \frac{b_0(k+1)}{4}\big[  \gamma k + 5\gamma + 2\rho - \frac{b_0(k+1)}{4} - (3\gamma +2\rho)\big] = \frac{b_0(k+1)}{4}[ (\gamma - \frac{b_0}{4})(k+1) + \gamma ] = \frac{b_0(k+1)}{16}[ (4\gamma - b_0)(k+1) + 4\gamma]$
%
Let us choose $b_0 := 2\gamma$.
Then, we have $\hat{\Pc}_k \geq \frac{\gamma^2  (k+1)(k+3)}{4}\norms{w^{k-1}}^2$.
Since $\hat{\Pc}_k \leq \hat{\Pc}_0 = a_0\norms{w^{-1}}^2 + b_0\iprods{w^{-1}, z^0 - y^0} + \norms{z^0 + t_0(y^0 - z^0) - x^{\star}}^2 +  c_0\norms{w^{-1} - \hat{w}^{-1}}^2$, and $y^0 = z^0 = x^0$ and $w^{-1} = \hat{w}^{-1}$, we get $\hat{\Pc}_k \leq \hat{\Pc}_0 = a_0\norms{w^{-1}}^2 + \norms{x^0 - x^{\star}}^2 = 2\gamma(5\gamma + 2\rho) \norms{w^0}^2 + \norms{x^0 - x^{\star}}^2$ since $w^{-1} = 2w^0$.
Combining these derivations, we finally get $\norms{w^{k-1}}^2 \leq \frac{4\norms{x^0 - x^{\star}}^2 + 8\gamma(5\gamma + 2\rho) \norms{w^0}^2}{\gamma^2(k+1)(k+3)}$, which implies \eqref{eq:AcPEG_convergence}.
If $J_{\eta T}$ is nonexpansive, then using $\norms{\Gc_{\eta}x^k} \leq \norms{w^k}$ from Lemma~\ref{le:P_lowerbound}, we obtain  \eqref{eq:AcPEG_convergence_b}.
\Eproof
\end{proof}
%%% End of proof.

%%%% Remark 3.
%\begin{remark}\label{re:constant_stepsize}
%It is important to note that our stepsize $\eta$ in both \eqref{eq:AcEG} and \eqref{eq:AcPEG} is constant, making our resolvent $J_{\eta T}$ to be fixed  at all the iterations.
%This is different from other works, including a very recent work \cite{sedlmayer2023fast}.
%\end{remark}

%%%%%%%%%%%%%%%%%%%%%%%%%%%%%%%%%%%%
%%% 4. The Extra-Anchored Gradient-Type Methods for Solving (cohMI).
%%%%%%%%%%%%%%%%%%%%%%%%%%%%%%%%%%%%
\beforesec
\section{The Extra-Anchored Gradient-Type Methods for Solving \eqref{eq:coMI}}\label{apdx:Halpern_AcEG}
\aftersec
In this section, we provide an alternative convergence analysis for Halpern's fixed-point-type variants of \eqref{eq:EG} and its past FBFS scheme, which is called EAG (extra-anchored gradient) method \cite{lee2021fast,yoon2021accelerated} and PEAG (past extra-anchored gradient) method in \cite{tran2021halpern}, respectively, where their names are  stemmed from \cite{yoon2021accelerated}.
These schemes have been extended to \eqref{eq:coMI} in \cite{cai2022accelerated} and \cite{cai2022baccelerated}, respectively.
Our algorithmic derivation as well as our analysis can be viewed as direct extensions of the results in \cite{lee2021fast,tran2021halpern,yoon2021accelerated}.
Moreover, together with the new analysis, our \ref{eq:cPEAG} is significantly different from that of \cite{cai2022baccelerated} since we replace $Fx^k$ by $Fy^{k-1}$ instead of replacing $Fy^k$ by $Fx^{k-1}$ as often seen in past extragradient-type methods.

%%%%%%%%%%%%%%%%%%%%%%%%%%%%%%%%%%%%%%%%%%%%
%%% 3.1. The extra-anchored gradient method.
%%%%%%%%%%%%%%%%%%%%%%%%%%%%%%%%%%%%%%%%%%%%
\beforesubsec
\subsection{The extra-anchored gradient method for solving \eqref{eq:coMI}}\label{subsec:EAG2}
\aftersubsec
We first derive a variant of the extra-anchored gradient method (EAG) from \cite{lee2021fast,yoon2021accelerated} to solve \eqref{eq:coMI}.
We call this method \eqref{eq:EAG} as in \cite{yoon2021accelerated}, which can be described as follows:
Starting from $x^0\in\R^p$, at each iteration $k \geq 0$, we update
\begin{equation}\label{eq:EAG}
\arraycolsep=0.2em
\left\{\begin{array}{lcl}
y^k        &:= &  x^k + \tau_k(x^0 - x^k) - ( \hat{\eta}_k - 2\rho(1-\tau_k)) (Fx^k + \xi^k), \vspace{1ex}\\
x^{k+1} &:= & x^k + \tau_k(x^0 - x^k) - \eta (Fy^k + \xi^{k+1}) + 2\rho(1-\tau_k)(Fx^k + \xi^k),
\end{array}\right.
\tag{EAG}
\end{equation}
where $\xi^k \in Tx^k$, $\tau_k \in (0, 1)$, $\eta  > 0$, and $\hat{\eta}_k > 0$ are  determined later.

This variant was studied in \cite{cai2022accelerated}. 
However, if $T = 0$, then \eqref{eq:EAG} reduces to the EAG variant in \cite{lee2021fast}.
%This variant was studied in \cite{cai2022accelerated}, and here we provide an alternative convergence analysis, which is different from \cite{cai2022accelerated}.
%
We introduce the following quantities to simplify our analysis:
\begin{equation}\label{eq:EAG_ex2}
w^k := Fx^k + \xi^k \in Fx^k + Tx^k \quad\text{and} \quad \hat{w}^k := Fy^{k-1} + \xi^k.
\end{equation}
Then, we can rewrite \eqref{eq:EAG} as follows:
\begin{equation}\label{eq:EAG_ex3}
\arraycolsep=0.2em
\left\{\begin{array}{lcl}
y^k        &:= &  x^k + \tau_k(x^0 - x^k) - ( \hat{\eta}_k - 2\rho(1-\tau_k)) w^k, \vspace{1ex}\\
x^{k+1} &:= & x^k + \tau_k(x^0 - x^k) - \eta \hat{w}^{k+1} + 2\rho(1-\tau_k)w^k.
\end{array}\right.
\end{equation}
The interpretation \eqref{eq:EAG_ex3} looks very similar to the fast extragradient scheme in \cite{lee2021fast} for root-finding problem \eqref{eq:coME}, where $w^k$ and $\hat{w}^{k+1}$ play the roles of $Fx^k$ and $Fy^k$, respectively in  \cite{lee2021fast}.

Since $x^{k+1}$ appears on both sides of the second line of \eqref{eq:EAG_ex3}, we will use the resolvent $J_{\eta T}$ to rewrite it as follows:
\begin{equation}\label{eq:EAG_impl}
\arraycolsep=0.2em
\left\{\begin{array}{lcl}
y^k        &:= &  x^k + \tau_k(x^0 - x^k) - ( \hat{\eta}_k - 2\rho(1-\tau_k)) (Fx^k + \xi^k), \vspace{1ex}\\
x^{k+1}  & := & J_{\eta T}\big( y^k - \eta Fy^k  +  \hat{\eta}_k(Fx^k + \xi^k)  \big), \vspace{1ex}\\
\xi^{k+1} & := &  \frac{1}{\eta}\big( y^k -  x^{k+1} - \eta Fy^k  +  \hat{\eta}_k(Fx^k + \xi^k)  \big),
\end{array}\right.
\end{equation}
where $\tau_k \in (0, 1)$ and $\xi^0 \in Tx^0$ is arbitrary.

To establish convergence of \eqref{eq:EAG}, we use the following Lyapunov function:
\begin{equation}\label{eq:EAG_potential_func}
\Vc_k :=  a_k\norms{w^k}^2 + b_k\iprods{w^k, x^k - x^0},
\end{equation}
where $a_k > 0$ and $b_k > 0$ are given parameters.
This function is similar to the one in \cite{lee2021fast,yoon2021accelerated} and is the same as the one in \cite{cai2022accelerated}.
The following theorem states the convergence of \eqref{eq:EAG}.

%%% Theorem 2.4.
\begin{theorem}\label{th:cEAG_convergence}
Assume that $F + T$ in \eqref{eq:coMI} is $\rho$-co-hypomonotone, $F$ is $L$-Lipschitz continuous such that $2L\rho < 1$, and $\zer{F+T}\neq\emptyset$.
Let $\sets{(x^k, y^k)}$ be generated by \eqref{eq:EAG} using
\begin{equation}\label{eq:EAG_param_update}
\tau_k := \frac{1}{k+2}, \quad \eta \in \left(2\rho,  \frac{1}{L}\right], \quad\text{and} \quad  \hat{\eta}_k := (1-\tau_k)\eta.
\end{equation}
Then, for all $k\geq 0$ and any $x^{\star}\in\zer{F+T}$, the following result holds:
\begin{equation}\label{eq:EAG_convergence1}
\norms{Fx^k + \xi^k}^2 \leq  \frac{4\norms{x^0 - x^{\star}}^2 + 2\eta(\eta - 2\rho)\norms{Fx^0 + \xi^0}^2}{(\eta - 2\rho)^2(k + 1)^2}, \quad\text{where}\quad \xi^k \in Tx^k.
\end{equation}
\revised{
Consequently, if $J_{\eta T}$ is nonexpansive, then for $\Gc_{\eta}$ defined by \eqref{eq:coMI_residual}, we have}
\begin{equation}\label{eq:EAG_convergence1b}
\revised{\norms{\Gc_{\eta}x^k}^2 \leq  \frac{4\norms{x^0 - x^{\star}}^2 + 2\eta(\eta - 2\rho)\norms{Fx^0 + \xi^0}^2}{(\eta - 2\rho)^2(k + 1)^2}.}
\end{equation}
\end{theorem}

%%% The proof of Theorem 2.4.
\begin{proof}
Since \eqref{eq:EAG} is equivalent to \eqref{eq:EAG_ex3}, from the second line of \eqref{eq:EAG}, we have
\begin{equation}\label{eq:EAG_proof1}
\arraycolsep=0.2em
\left\{\begin{array}{lcl}
x^{k+1} - x^k & = &  - \tau_k(x^k - x^0) - \eta \hat{w}^{k+1} + 2\rho(1-\tau_k)w^k \vspace{1ex}\\
x^{k+1} - x^k & = & -\tfrac{\tau_k}{1-\tau_k}(x^{k+1} - x^0) - \tfrac{\eta}{1-\tau_k}\hat{w}^{k+1} + 2\rho w^k.
\end{array}\right.
\end{equation}
Next, since  $F+T$ is $\rho$-co-hypomonotone and $w^k \in Fx^k + Tx^k$, we have $\iprods{w^{k+1} - w^k, x^{k+1} - x^k} + \rho\norms{w^{k+1} - w^k}^2 \geq 0$.
This relation leads to
\begin{equation*} 
\arraycolsep=0.2em
\begin{array}{lcl}
0 &\leq & \iprods{w^{k+1}, x^{k+1} - x^k} - \iprods{w^k, x^{k+1} - x^k}  + \rho\norms{w^{k+1} - w^k}^2 \vspace{1ex}\\
&\overset{\tiny\eqref{eq:EAG_proof1}}{=} & \tau_k\iprods{w^k, x^k - x^0} - \frac{\tau_k}{1-\tau_k}\iprods{w^{k+1}, x^{k+1} - x^0} + \eta \iprods{w^k, \hat{w}^{k+1}} - 2\rho(1-\tau_k)\norms{w^k}^2 \vspace{1ex}\\
&& -  {~}  \frac{\eta}{1-\tau_k}\iprods{w^{k+1}, \hat{w}^{k+1}} + 2\rho\iprods{w^{k+1}, w^k} + \rho\norms{w^{k+1} - w^k}^2.
\end{array}
\end{equation*}
Multiplying this expression by $\frac{b_k}{\tau_k}$, rearranging the result, and using $b_{k+1} = \frac{b_k}{1-\tau_k}$, we obtain the following expression 
\begin{equation*} 
\arraycolsep=0.2em
\begin{array}{lcl}
\Tc_{[1]} &:= & b_k\iprods{w^k, x^k - x^0} - b_{k+1}\iprods{w^{k+1}, x^{k+1} - x^0} \vspace{1ex} \\ 
&\geq & \frac{\eta b_{k+1}}{\tau_k}\iprods{w^{k+1} - w^k, \hat{w}^{k+1}} + \eta b_{k+1} \iprods{\hat{w}^{k+1}, w^k} - \frac{\rho b_k}{\tau_k}\norms{w^{k+1}}^2 + \frac{\rho b_k(1 - 2\tau_k)}{\tau_k}\norms{w^k}^2.
\end{array}
\end{equation*}
Adding $a_k\norms{w^k}^2 - a_{k+1}\norms{w^{k+1}}^2$ to $\Tc_{[1]}$ and  using $\Vc_k$ from \eqref{eq:EAG_potential_func}, we can show that
\begin{equation}\label{eq:EAG_proof2}
\hspace{-3ex}
\arraycolsep=0.2em
\begin{array}{lcl}
\Vc_k - \Vc_{k+1} &= & a_k\norms{w^k}^2 - a_{k+1}\norms{w^{k+1}}^2 + b_k\iprods{w^k, x^k - x^0} - b_{k+1}\iprods{w^{k+1}, x^{k+1} - x^0} \vspace{1ex}\\
&\geq &  \big(a_k + \frac{\rho b_k(1 - 2\tau_k)}{\tau_k} \big) \norms{w^k}^2 - \big(a_{k+1} + \frac{\rho b_k}{\tau_k}\big)\norms{w^{k+1}}^2  \vspace{1ex}\\
&& +  {~}  \frac{\eta b_{k+1} }{\tau_k}\iprods{w^{k+1} - w^k, \hat{w}^{k+1}} + \eta b_{k+1} \iprods{\hat{w}^{k+1}, w^k}.
\end{array}
\hspace{-5ex}
\end{equation}
Next, from \eqref{eq:EAG_ex3}, we have $x^{k+1} - y^k = -\eta \hat{w}^{k+1} + \hat{\eta}_kw^k$.
By the $L$-Lipschitz continuity of $F$, we have $\norms{w^{k+1} - \hat{w}^{k+1}}^2 = \norms{Fx^{k+1} - Fy^k}^2 \leq L^2\norms{x^{k+1} - y^k}^2 = L^2\norms{\eta\hat{w}^{k+1} - \hat{\eta}_kw^k}^2$.
Expanding this inequality and rearranging the result, we get
\begin{equation*} 
\arraycolsep=0.2em
\begin{array}{lcl}
0 & \geq & \norms{w^{k+1}}^2 + (1 - L^2\eta^2)\norms{\hat{w}^{k+1}}^2 - 2\iprods{w^{k+1} - w^k, \hat{w}^{k+1}} \vspace{1ex}\\
&& - {~} 2\big(1 - L^2\eta\hat{\eta}_k)\iprods{\hat{w}^{k+1}, w^k} - L^2\hat{\eta}_k^2\norms{w^k}^2.
\end{array}
\end{equation*}
Multiplying this inequality by $\frac{\eta b_{k+1} }{2\tau_k}$ and adding the result to \eqref{eq:EAG_proof2}, we get
\begin{equation}\label{eq:EAG_proof3}
\hspace{-4ex}
\arraycolsep=0.2em
\begin{array}{lcl}
\Vc_k - \Vc_{k+1} &\geq & \left(a_k  - \frac{L^2\eta \hat{\eta}_k^2b_{k+1} - 2\rho b_k(1-2\tau_k)}{2\tau_k} \right)\norms{w^k}^2 + \left(\frac{\eta b_{k+1} - 2\rho b_k}{2\tau_k}  - a_{k+1} \right) \norms{w^{k+1}}^2  \vspace{2ex}\\
&& +  {~}  \frac{\eta(1 - L^2\eta^2)b_{k+1}}{2\tau_k}\norms{\hat{w}^{k+1}}^2  -  \frac{\eta(1 - \tau_k - L^2\eta\hat{\eta}_k) b_{k+1} }{\tau_k} \iprods{\hat{w}^{k+1}, w^k}.
\end{array}
\hspace{-7ex}
\end{equation}
Let us choose $\tau_k := \frac{1}{k + 2}$ and $\hat{\eta}_k := (1-\tau_k)\eta$ as in \eqref{eq:EAG_param_update}.
We also choose $a_{k+1} := \frac{b_{k+1}[ \eta  - 2\rho(1-\tau_k)]}{2\tau_k} = \frac{[(\eta - 2\rho)(k + 2) + 2\rho] b_{k+1}}{2}$ provided that $\eta > 2\rho$.
By the update of $b_k$ as $b_{k+1} = \frac{b_k}{1-\tau_k} $, we can easily show that $b_k = b_0(k+1)$ for $k\geq 0$.
Therefore, we get $a_k = \frac{b_0[(\eta - 2\rho)(k+1) + 2\rho](k+1) }{2}$.

Using these parameters and assuming that $L\eta \leq 1$, we can simplify \eqref{eq:EAG_proof3} as
\begin{equation}\label{eq:EAG_proof3_b}
\arraycolsep=0.2em
\begin{array}{lcl}
\Vc_k - \Vc_{k+1} &\geq & \frac{\eta(1 - L^2\eta^2)b_{k+1}}{2\tau_k}\norms{\hat{w}^{k+1} - (1-\tau_k)w^k}^2 \geq 0.
\end{array}
\end{equation}
Finally, since $0 \in Fx^{\star} + Tx^{\star}$, we have $\iprods{w^k, x^k - x^{\star}} \geq -\rho\norms{w^k}^2$.
Using this inequality, the definition of $\Vc_k$, and Young's inequality, we can show that
\begin{equation*}
\arraycolsep=0.2em
\begin{array}{lcl}
\Vc_k & = & a_k\norms{w^k}^2 + b_k\iprods{w^k, x^{\star} - x^0} + b_k\iprods{w^k, x^k - x^{\star}} \vspace{1ex}\\
&\geq &  ( a_k - \rho b_k) \norms{w^k}^2 - b_k\norms{w^k}\norms{x^0 - x^{\star}} \vspace{1ex}\\
&\geq &  \left( a_k - \rho b_k - \frac{(\eta - 2\rho) b_k^2}{4b_0} \right) \norms{w^k}^2 - \frac{b_0}{\eta - 2\rho}\norms{x^0 - x^{\star}}^2 \vspace{1ex}\\
&= & \frac{b_0(\eta - 2\rho)(k + 1)^2}{4}\norms{w^k}^2 - \frac{b_0}{\eta- 2\rho}\norms{x^0 - x^{\star}}^2.
\end{array}
\end{equation*}
Moreover, we also have $\Vc_0 = a_0\norms{w^0}^2 + b_0\iprods{w^0, x^0 - x^0} = \frac{\eta b_0}{2}\norms{w^0}^2$.
Combining all these derivations and using \eqref{eq:EAG_proof3_b} to get $\Vc_k \leq \Vc_0 = \frac{\eta b_0}{2}\norms{w^0}^2$, we obtain \eqref{eq:EAG_convergence1}.
The conditions $\eta > 2\rho$ and $L\eta \leq 1$ imply that $2\rho < \eta \leq \frac{1}{L}$, which is satisfied if $2L\rho < 1$.
\revised{The bound \eqref{eq:EAG_convergence1b} is a consequence of \eqref{eq:EAG_convergence1} and $\norms{\Gc_{\eta}x^k} \leq \norms{Fx^k + \xi^k}$ in Lemma~\ref{le:P_lowerbound}.}
\Eproof
\end{proof}
%%% End of the proof.

Clearly, the condition on $L$ and $\rho$, and the convergence rate bounds \eqref{eq:AcEG_convergence} and \eqref{eq:EAG_convergence1} of \eqref{eq:AcEG} are similar to the ones in \eqref{eq:EAG}, respectively.
As shown in \cite{lee2021fast} for the case $T= 0$, these convergence rates are tight.

%%%%%%%%%%%%%%%%%%%%%%%%%%%%%%%%%%%%%%%%%%%%
%%% 3.2. The past- extra-anchored gradient method.
%%%%%%%%%%%%%%%%%%%%%%%%%%%%%%%%%%%%%%%%%%%%
\beforesubsec
\subsection{The past extra-anchored gradient method for solving \eqref{eq:coMI}}\label{subsec:PEAG2}
\aftersubsec
We now develop a variant of the past extra-anchored gradient method (PEAG) from \cite{tran2021halpern} to solve \eqref{eq:coMI}.
We still call this method \eqref{eq:cPEAG} and it is presented as follows:
Starting from $x^0\in\R^p$, we set $y^{-1} := x^0$, and for $k \geq 0$ we update
\begin{equation}\label{eq:cPEAG}
\arraycolsep=0.2em
\left\{\begin{array}{lcl}
y^k        &:= &  x^k + \tau_k(x^0 - x^k) - ( \hat{\eta}_k - \beta_k) (Fy^{k-1} + \xi^k), \vspace{1ex}\\
x^{k+1} &:= & x^k + \tau_k(x^0 - x^k) - \eta (Fy^k + \xi^{k+1}) + \beta_k(Fy^{k-1} + \xi^k), 
\end{array}\right.
\tag{PEAG}
\end{equation}
where $\xi^k \in Tx^k$ is an arbitrary element of $Tx^k$, $\tau_k \in (0, 1)$, $\eta > 0$, $\hat{\eta}_k > 0$, and $\beta_k > 0$ are given parameters, which will be determined later.

Note that if $T = 0$, then \eqref{eq:cPEAG} reduces to the past extra-anchored gradient scheme in \cite{tran2021halpern}.
This scheme can be considered as a modification of \eqref{eq:EAG} by replacing $Fx^k$ by $Fy^{k-1}$ using Popov's idea in \cite{popov1980modification}.
However, it is different from \eqref{eq:PEG} as we do not replace $Fy^k$ by $Fx^{k-1}$.
Compared to \cite{cai2022baccelerated}, since their method is rooted from the reflected gradient method in \cite{malitsky2015projected}, it is also different from \eqref{eq:cPEAG} though both methods have the same per-iteration complexity and convergence rate.
Moreover, our condition $2\sqrt{34} L\rho < 1$ in Theorem~\ref{th:cPEAG_convergence} below is better than $12\sqrt{3}L\rho < 1$ in  \cite{cai2022baccelerated}.
Our proof below is also different from  \cite{cai2022baccelerated}, where we follow the proof technique in \cite{tran2021halpern}.

Let us rewrite \eqref{eq:cPEAG} into a different form.
We introduce  $w^k := Fx^k + \xi^k \in Fx^k + Tx^k$ and $\hat{w}^k := Fy^{k-1} + \xi^k$ as in \eqref{eq:EAG}.
Then, we can rewrite \eqref{eq:cPEAG} as 
\begin{equation}\label{eq:cPEAG_ex3}
\arraycolsep=0.2em
\left\{\begin{array}{lcl}
y^k        &:= &  x^k + \tau_k(x^0 - x^k) - ( \hat{\eta}_k  - \beta_k) \hat{w}^k, \vspace{1ex}\\
x^{k+1} &:= & x^k + \tau_k(x^0 - x^k) - \eta \hat{w}^{k+1} + \beta_k \hat{w}^k.
\end{array}\right.
\end{equation}
This scheme has the same form as \cite{tran2021halpern} but with different parameters, where $\hat{w}^k$ plays the same role as $Fy^k$ in \cite{tran2021halpern}.

Note that the second line of \eqref{eq:cPEAG} is an implicit update since $x^{k+1}$ is also in $\xi^{k+1} \in Tx^{k+1}$ on the right-hand side of this line.
To resolve this implicit issue, we rewrite it as follows:
\begin{equation}\label{eq:EAG_impl}
\arraycolsep=0.2em
\left\{\begin{array}{lcl}
y^k        &:= &  x^k + \tau_k(x^0 - x^k) - ( \hat{\eta}_k  - \beta_k)\hat{w}^k, \vspace{1ex}\\
x^{k+1}  & := & J_{\eta T}\big( y^k - \eta Fy^k  +  \hat{\eta}_k\hat{w}^k  \big), \vspace{1ex}\\
\hat{w}^{k+1} & := &  \frac{1}{\eta}\big( y^k   - x^{k+1} +  \hat{\eta}_k\hat{w}^k \big), 
\end{array}\right.
\end{equation}
where $x^0 \in \R^p$ is given, $y^{-1}  := x^0$, and $\hat{w}^0 := w^0 \in Fx^0 + Tx^0$ is arbitrary.

To establish convergence of \eqref{eq:cPEAG}, we use the following Lyapunov function:
\begin{equation}\label{eq:cPEAG_potential_func}
\hat{\Vc}_k :=   a_k\norms{w^k}^2 + b_k\iprods{w^k, x^k - x^0} + c_k\norms{w^k - \hat{w}^k}^2,
\end{equation}
where $w^k := \hat{w}^k + Fx^k - Fy^{k-1} = Fx^k + \xi^k$, $a_k > 0$, $b_k > 0$, and $c_k > 0$ are given parameters.
Then, the following theorem states the convergence rate of \eqref{eq:cPEAG}.

%%% Theorem 2.4.
\begin{theorem}\label{th:cPEAG_convergence}
Assume that $F+ T$ in \eqref{eq:coMI} is $\rho$-co-hypomonotone, $F$ is $L$-Lipschitz continuous such that $2\sqrt{34}L\rho < 1$, and $\zer{F + T}\neq\emptyset$.  
Let $\eta := \frac{1}{L}\sqrt{\frac{2}{17}}$ be a given stepsize, and $\sets{(x^k, y^k)}$ be generated by \eqref{eq:cPEAG} using the following parameters:
\begin{equation}\label{eq:cPEAG_param_update}
\tau_k := \frac{1}{k+2}, \quad \beta_k := \frac{4\rho(1-\tau_k)}{1 + \tau_k},  \quad\text{and} \quad  \hat{\eta}_k := (1-\tau_k)\eta.
\end{equation}
Then, for all $k\geq 0$ and any $x^{\star}\in\zer{F+T}$, we have 
\begin{equation}\label{eq:cPEAG_convergence1}
\norms{Fx^k + \xi^k}^2  \leq    \frac{1}{(k+1)^2}\left[ \frac{4}{(\eta - 4\rho)^2}\norms{x^0 - x^{\star}}^2 +  \frac{2(3\eta - 2\rho)}{3(\eta - 4\rho)}  \norms{Fx^0 + \xi^0}^2\right],
\end{equation}
where $\xi^k \in Tx^k$ and $\xi^0 \in Tx^0$.

\revised{
Consequently, if $J_{\eta T}$ is nonexpansive, then for $\Gc_{\eta}$ defined by \eqref{eq:coMI_residual}, we have}
\begin{equation}\label{eq:cPEAG_convergence1b}
\revised{ \norms{\Gc_{\eta}x^k}^2 \leq \frac{1}{(k+1)^2}\Big[ \frac{4}{(\eta - 4\rho)^2}\norms{x^0 - x^{\star}}^2 +  \frac{2(3\eta - 2\rho)}{3(\eta - 4\rho)}  \norms{Fx^0 + \xi^0}^2 \Big].}
\end{equation}
\end{theorem}

%%% The proof of Theorem 2.4.
\begin{proof}%%[The proof of Theorem~\ref{th:cPEAG_convergence}]
Since \eqref{eq:cPEAG} is equivalent to \eqref{eq:cPEAG_ex3}, from \eqref{eq:cPEAG_ex3}, we can easily show that
\begin{equation}\label{eq:cPEAG_proof1}
\hspace{-2ex}
\arraycolsep=0.2em
\left\{\begin{array}{lcl}
x^{k+1} - x^k & = &  -\tau_k(x^k - x^0) - \eta \hat{w}^{k+1} + \beta_kw^k + \beta_k(\hat{w}^k - w^k) \vspace{1ex}\\
x^{k+1} - x^k & = & -\tfrac{\tau_k}{1-\tau_k}(x^{k+1} - x^0) - \tfrac{\eta}{1-\tau_k}\hat{w}^{k+1}  + \frac{\beta_k}{1-\tau_k}w^k +  \frac{\beta_k}{1-\tau_k}(\hat{w}^k - w^k).
\end{array}\right.
\hspace{-2ex}
\end{equation}
Next, since $F+T$ is $\rho$-co-hypomonotone and $w^k = Fx^k + \xi^k \in Fx^k + Tx^k$, we have $\iprods{w^{k+1} - w^k, x^{k+1} - x^k} + \rho\norms{w^{k+1} - w^k}^2 \geq 0$.
Using \eqref{eq:cPEAG_proof1} into this inequality, we can expand it as
\begin{equation*} 
\arraycolsep=0.2em
\begin{array}{lcl}
\Tc_{[1]} &:= & \tau_k\iprods{w^k, x^k - x^0} - \frac{\tau_k}{1-\tau_k}\iprods{w^{k+1}, x^{k+1} - x^0} \vspace{1ex}\\
&\geq &  \frac{\eta}{1-\tau_k}\iprods{w^{k+1}, \hat{w}^{k+1}}   - \eta \iprods{w^k, \hat{w}^{k+1}}  - \rho\norms{w^{k+1} - w^k}^2 + \beta_k\norms{w^k}^2  \vspace{1ex}\\
&& - {~}   \frac{\beta_k}{1-\tau_k}\iprods{w^{k+1}, w^k} - \frac{\beta_k}{1-\tau_k}\iprods{w^{k+1} - (1-\tau_k)w^k, \hat{w}^k - w^k}.
\end{array}
\end{equation*}
Using Young's inequality and choosing $\beta_k := \frac{4\rho(1-\tau_k)}{1 + \tau_k}$, we can further derive
\begin{equation*} 
\arraycolsep=0.2em
\begin{array}{lcl}
\Tc_{[1]} &:= & \tau_k\iprods{w^k, x^k - x^0} - \frac{\tau_k}{1-\tau_k}\iprods{w^{k+1}, x^{k+1} - x^0} \vspace{1ex}\\
&\geq&  \frac{\eta}{1-\tau_k}\iprods{w^{k+1}, \hat{w}^{k+1}} - \eta \iprods{w^k, \hat{w}^{k+1}} -  \rho\norms{w^{k+1} - w^k}^2 + \beta_k\norms{w^k}^2  \vspace{1ex}\\
&& -  {~}    \frac{\beta_k}{1-\tau_k}\iprods{w^{k+1}, w^k}  - \frac{\beta_k}{4(1-\tau_k)}\norms{w^{k+1} - (1-\tau_k)w^k}^2 - \frac{\beta_k}{1-\tau_k} \norms{\hat{w}^k - w^k}^2 \vspace{1ex}\\
& = & \frac{\eta}{1-\tau_k}\iprods{w^{k+1}, \hat{w}^{k+1}} - \eta \iprods{w^k, \hat{w}^{k+1}} -  \left[ \frac{(1+\tau_k)\beta_k}{2(1-\tau_k)} - 2\rho \right] \iprods{w^{k+1}, w^k} \vspace{1ex}\\
&& - {~} \left[  \frac{\beta_k}{4(1-\tau_k)} + \rho \right] \norms{w^{k+1}}^2 + \left[ \beta_k -  \frac{\beta_k(1-\tau_k)}{4} - \rho \right] \norms{w^k}^2   - \frac{\beta_k}{1-\tau_k} \norms{\hat{w}^k - w^k}^2  \vspace{1ex}\\
&= &   \frac{\eta }{1-\tau_k}\iprods{w^{k+1}, \hat{w}^{k+1}}  - \eta \iprods{w^k, \hat{w}^{k+1}} - \frac{\rho(2 + \tau_k)}{1+\tau_k} \norms{w^{k+1}}^2  + \frac{\rho (2 -  3\tau_k - \tau_k^2) }{1 + \tau_k} \norms{w^k}^2 \vspace{1ex}\\
&& - {~}  \frac{4\rho}{1 + \tau_k} \norms{\hat{w}^k - w^k}^2.
\end{array}
\end{equation*}
Multiplying $\Tc_{[1]}$ by $\frac{b_k}{\tau_k}$, rearranging the result, and using $b_{k+1} = \frac{b_k}{1-\tau_k}$, we can show that
\begin{equation*} 
\arraycolsep=0.2em
\begin{array}{lcl}
\hat{\Tc}_{[1]} &:= & b_k\iprods{w^k, x^k - x^0} - b_{k+1}\iprods{w^{k+1}, x^{k+1} - x^0} \vspace{1ex}\\  
&\geq & \frac{b_{k+1}\eta}{\tau_k}\iprods{w^{k+1} - w^k, \hat{w}^{k+1}}  + b_{k+1}\eta\iprods{w^k, \hat{w}^{k+1}}   - \frac{\rho b_k(2 + \tau_k)}{\tau_k(1 + \tau_k)} \norms{w^{k+1}}^2 \vspace{1ex}\\
&& + {~}  \frac{\rho b_k ( 2 - 3\tau_k - \tau_k^2 ) }{\tau_k(1+\tau_k)} \norms{w^k}^2 -  \frac{4 \rho b_k}{\tau_k(1+\tau_k)} \norms{\hat{w}^k - w^k}^2.
\end{array}
\end{equation*}
Adding $a_k\norms{w^k}^2 - a_{k+1}\norms{w^{k+1}}^2 + c_k\norms{w^k - \hat{w}^k}^2 - c_{k+1}\norms{w^{k+1} - \hat{w}^{k+1}}^2$ to both sides of $\hat{\Tc}_{[1]}$, then  using $\hat{\Vc}_k$ defined by \eqref{eq:cPEAG_potential_func}, we can show that
\begin{equation}\label{eq:cPEAG_proof2}
\arraycolsep=0.2em
\begin{array}{lcl}
\hat{\Vc}_k - \hat{\Vc}_{k+1} 
%&= & a_k\norms{w^k}^2 - a_{k+1}\norms{w^{k+1}}^2 + b_k\iprods{w^k, x^k - x^0} - b_{k+1}\iprods{w^{k+1}, x^{k+1} - x^0} \vspace{1ex}\\
%&& + {~} c_k\norms{w^k - \hat{w}^k}^2 - c_{k+1}\norms{w^{k+1} - \hat{w}^{k+1}}^2 \vspace{1ex}\\
&\geq &  \left[ a_k +  \frac{\rho b_k ( 2 - 3 \tau_k - \tau_k^2 ) }{\tau_k(1 + \tau_k)}  \right] \norms{w^k}^2 - \left[ a_{k+1} +  \frac{\rho b_k(2 + \tau_k)}{\tau_k(1 + \tau_k)}  \right]\norms{w^{k+1}}^2  \vspace{1ex}\\
&& + \left[c_k -  \frac{4 \rho b_k}{\tau_k(1 + \tau_k)}  \right]\norms{w^k - \hat{w}^k}^2 - c_{k+1}\norms{w^{k+1} - \hat{w}^{k+1}}^2\vspace{1ex}\\
&& +  {~}  \frac{\eta b_{k+1} }{\tau_k}  \iprods{w^{k+1} - w^k, \hat{w}^{k+1}} + \eta b_{k+1} \iprods{\hat{w}^{k+1}, w^k}.
\end{array}
\end{equation}
Now, from \eqref{eq:cPEAG_ex3}, we have $x^{k+1} - y^k = -\eta\hat{w}^{k+1} + \hat{\eta}_k\hat{w}^k$.
Therefore, using the $L$-Lipschitz continuity of $F$ and Young's inequality, we can show that
\begin{equation*} 
\arraycolsep=0.2em
\begin{array}{lcl}
\norms{w^{k+1} - \hat{w}^{k+1}}^2 & = & \norms{Fx^{k+1} - Fy^k}^2 \leq L^2\norms{x^{k+1} - y^k}^2 = L^2\norms{\eta \hat{w}^{k+1} - \hat{\eta}_k\hat{w}^k}^2 \vspace{1ex}\\
&\leq & 2L^2\norms{\eta\hat{w}^{k+1} - \hat{\eta}_kw^k}^2 + 2L^2\hat{\eta}_k^2\norms{w^k - \hat{w}^k}^2.
\end{array}
\end{equation*}
Multiplying this inequality by $(1+\omega)$ for some $\omega > 0$, defining $M := 2(1 + \omega)L^2$, and expanding and rearranging the result, we obtain
\begin{equation*} 
\arraycolsep=0.2em
\begin{array}{lcl}
0 & \geq & \omega\norms{w^{k+1} - \hat{w}^{k+1}}^2 + \norms{w^{k+1}}^2 + (1 - M\eta^2)\norms{\hat{w}^{k+1}}^2 - 2 \iprods{w^{k+1} - w^k, \hat{w}^{k+1}} \vspace{1ex}\\
&& - {~} 2(1 - M\eta\hat{\eta}_k)\iprods{\hat{w}^{k+1}, w^k} - M\hat{\eta}_k^2\norms{w^k}^2 - M\hat{\eta}_k^2\norms{w^k - \hat{w}^k}^2.
\end{array}
\end{equation*}
Multiplying this inequality by $\frac{\eta b_{k+1} }{2\tau_k}$ and adding the result to \eqref{eq:cPEAG_proof2}, we get
\begin{equation}\label{eq:cPEAG_proof3}
\arraycolsep=0.2em
\begin{array}{lcl}
\hat{\Vc}_k - \hat{\Vc}_{k+1} &\geq & \left[ c_k - \frac{4\rho b_k}{\tau_k(1 + \tau_k)}  - \frac{Mb_{k+1}\eta\hat{\eta}_k^2}{2\tau_k} \right]\norms{w^k - \hat{w}^k}^2 \vspace{1ex}\\
&& + {~} \left( \frac{\omega\eta b_{k+1}}{2\tau_k} -  c_{k+1} \right) \norms{w^{k+1} - \hat{w}^{k+1}}^2\vspace{1ex}\\
&& + {~} \left[ a_k + \frac{\rho b_k ( 2  - 3 \tau_k - \tau_k^2 ) }{\tau_k(1 + \tau_k)}  - \frac{Mb_{k+1}\eta\hat{\eta}_k^2}{2\tau_k} \right] \norms{w^k}^2 \vspace{1ex}\\
&& + {~} \left[ \frac{\eta b_{k+1} }{2\tau_k} - \frac{\rho b_k(2 + \tau_k)}{\tau_k(1 + \tau_k)}  - a_{k+1} \right]\norms{w^{k+1}}^2 \vspace{1ex}\\
&& + {~}  \frac{\eta(1-M\eta^2)b_{k+1}}{2\tau_k}\norms{\hat{w}^{k+1}}^2  -  \frac{\eta\left(1 -  \tau_k - M\eta\hat{\eta}_k\right)b_{k+1}}{\tau_k}  \iprods{\hat{w}^{k+1}, w^k}.
\end{array}
\end{equation}
Let us choose $\tau_k := \frac{1}{k+2}$, $\hat{\eta}_k := (1 - \tau_k)\eta$, $a_k := \frac{b_k}{2}\left(\eta(k+1) -  4\rho k + \frac{2\rho(k-1)}{k+3} \right)$, and $c_k := \frac{b_k}{2}\left( M\eta^3(k+1) +  \frac{8\rho(k+2)^2}{k+3} \right)$.
Then, utilizing these choices and $b_{k+1}(1-\tau_k) = b_k$, \eqref{eq:cPEAG_proof3} reduces to
\begin{equation}\label{eq:cPEAG_proof3_b}
\arraycolsep=0.2em
\begin{array}{lcl}
\hat{\Vc}_k - \hat{\Vc}_{k+1} &\geq & \frac{\eta(1-M\eta^2)b_{k+1}}{2\tau_k}\norms{\hat{w}^{k+1} - (1-\tau_k)w^k}^2 + \frac{2\rho b_{k+1}(k+2)}{(k+3)(k+4)}\norms{w^{k+1}}^2 \vspace{1ex}\\
&& + {~} \frac{b_{k+1}(k+2)}{2}\left( \omega\eta -   M\eta^3 - \frac{8\rho(k+3)^2}{(k+2)(k+4)}  \right) \norms{w^{k+1} - \hat{w}^{k+1}}^2. 
\end{array}
\end{equation}
%\begin{equation}\label{eq:cPEAG_proof3_b}
%\hspace{-0.75ex}
%\arraycolsep=0.2em
%\begin{array}{lcl}
%\hat{\Vc}_k - \hat{\Vc}_{k+1} &\geq & \frac{\eta(1-M\eta^2)b_{k+1}}{2\tau_k}\norms{\hat{w}^{k+1} - (1-\tau_k)w^k}^2 + \frac{2\rho b_{k+1}}{2(k+4)}\norms{w^{k+1}}^2 \vspace{1ex}\\
%&& + {~} \left[ c_k - \frac{b_k}{2}\left( M\eta^3(k+1) +  \frac{8\rho(k+2)^2}{k+3} \right) \right]\norms{w^k - \hat{w}^k}^2 \vspace{1ex}\\
%&& + {~} \left( \frac{\omega\eta b_{k+1}(k+2)}{2} -  c_{k+1} \right) \norms{w^{k+1} - \hat{w}^{k+1}}^2. 
%\end{array}
%\hspace{-6ex}
%\end{equation}
Clearly, we need to choose $\omega$ such that $\omega \eta \geq M\eta^3 + \frac{8\rho(k+3)^2}{(k+2)(k+4)}$, which holds if $\omega \eta \geq M\eta^3 + 9\rho$.
Moreover, to guarantee $a_k > 0$, we need to choose $\eta > 4\rho$.
Therefore, if the following three conditions hold:
\begin{equation}\label{eq:cPEAG_proof4}
2(1+\omega)L^2\eta^2 \leq 1, \quad  \omega \eta \geq 2(1+\omega)L^2\eta^3 + 9\rho, \quad \text{and} \quad \eta > 4\rho,
\end{equation}
then, from \eqref{eq:cPEAG_proof3_b}, we have $\hat{\Vc}_{k+1} \leq \hat{\Vc}_k$ for all $k\geq 0$.

For simplicity of our analysis, let us choose $\omega := \frac{13}{4}$ and $\eta := \frac{1}{L\sqrt{2(1+\omega)}} = \frac{\sqrt{2}}{\sqrt{17}L}$.
Then, the last two conditions of \eqref{eq:cPEAG_proof4} become $L\rho \leq \frac{\omega - 1}{9\sqrt{2(1+\omega)}}  = \frac{1}{2\sqrt{34}}$ and $L\rho < \frac{1}{4\sqrt{2(1+\omega)}} = \frac{1}{2\sqrt{34}}$, respectively.
Therefore,  if $2\sqrt{34}L\rho < 1$, then all the  conditions of \eqref{eq:cPEAG_proof4} hold.
%Therefore, we have $\eta = \frac{\sqrt{2}}{\sqrt{17}L}$.

%The last two conditions $\omega \eta \geq 2(1+\omega)L^2\eta^3 + 9\rho$ and $2(1+\omega)L^2\eta^2 \leq 1$ lead to the choice of $\omega$ as $\frac{2L^2\eta^3 + 9\rho}{\eta(1 - 2L^2\eta^2)} \leq \omega \leq \frac{1 - 2L^2\eta^2}{2L^2\eta^2}$ provided that $2L^2\eta^2 < 1$.
%This leads to the range of $\eta$ as $0 < \frac{\sqrt{81L^2\rho^2 + 4} - 8L\rho}{4L} \leq \eta \leq \frac{\sqrt{81L^2\rho^2 + 4} + 8L\rho}{4L}$ and $\eta < \frac{1}{\sqrt{2}L}$.
%However, from the choice of $a_k$, to guarantee that $a_k > 0$, we require $\eta > 4\rho$.
%Combining all the conditions, on $\eta$, $\rho$, and $L$, we require $4\sqrt{2}L\rho < 1$.
%In this case, we choose $\eta$ such that $\frac{\sqrt{81L^2\rho^2 + 4} - 8L\rho}{4L} \leq \eta \leq \frac{\sqrt{81L^2\rho^2 + 4} + 8L\rho}{4L}$, $\eta > 4\rho$, and $\eta < \frac{1}{\sqrt{2}L}$, which always exists.
%More explicitly, if $L\rho < \frac{2}{3\sqrt{55}}$, then we choose $\frac{\sqrt{81L^2\rho^2 + 4} - 8L\rho}{4L} \leq \eta \leq \frac{\sqrt{81L^2\rho^2 + 4} + 8L\rho}{4L}$.
%If $\frac{2}{3\sqrt{55}} \leq L\rho < \frac{1}{9\sqrt{2}}$, then we choose $4\rho < \eta < \frac{\sqrt{81L^2\rho^2 + 4} + 8L\rho}{4L}$.
%If $\frac{1}{9\sqrt{2}} \leq L\rho < \frac{1}{4\sqrt{2}}$, then we choose $4\rho < \eta < \frac{1}{\sqrt{2}L}$.
%For simplicity, if we choose $\eta := \frac{1}{2L}$ and impose $8L\rho < 1$, then all the conditions above hold.

Since $\tau_k := \frac{1}{k+2}$, we can easily show that $b_k = b_0(k+1)$ for some $b_0 > 0$.
Moreover, since $0 \in Fx^{\star} + Tx^{\star}$, we have $\iprods{w^k, x^k - x^{\star}} \geq -\rho\norms{w^k}^2$.
Using this inequality, the definition \eqref{eq:cPEAG_potential_func} of $\hat{\Vc}_k$, Young's inequality, $b_k = b_0(k+1)$, and the choice of $a_k$, we can show that
\begin{equation*}
\arraycolsep=0.2em
\begin{array}{lcl}
\hat{\Vc}_k &\geq &  \left( a_k - b_k\rho - \frac{(\eta - 4\rho)b_k^2}{4b_0} \right) \norms{w^k}^2 - \frac{b_0}{\eta - 4\rho}\norms{x^0 - x^{\star}}^2 \vspace{1ex}\\
&\geq & \frac{b_0(\eta - 4\rho)(k+1)^2}{4} \norms{w^k}^2 - \frac{b_0}{\eta - 4\rho}\norms{x^0 - x^{\star}}^2.
\end{array}
\end{equation*}
Since $y^{-1} = x^0$, we have $\hat{\Vc}_0 = a_0\norms{w^0}^2 + b_0\iprods{w^0, x^0 - x^0} + c_0\norms{w^0 - \hat{w}^0}^2 = a_0\norms{w^0}^2 = b_0\left(\frac{\eta}{2} - \frac{\rho}{3}\right)\norms{Fx^0 + \xi^0}^2$.
Therefore, we get $\hat{\Vc}_k \leq \hat{\Vc}_0 = b_0 \left(\frac{\eta}{2} - \frac{\rho}{3}\right)\norms{Fx^0 + \xi^0}^2$ from \eqref{eq:cPEAG_proof3_b}.
Combining this  and the lower bound of $\hat{\Vc}_k$ above, we obtain
\begin{equation*}
\arraycolsep=0.2em
\begin{array}{lcl}
\norms{w^k}^2 & \leq &  \frac{1}{(k+1)^2}\left[ \frac{4}{(\eta - 4\rho)^2}\norms{x^0 - x^{\star}}^2 +  \frac{2(3\eta - 2\rho)}{3(\eta - 4\rho)}  \norms{Fx^0 + \xi^0}^2\right],
\end{array}
\end{equation*}
which is exactly \eqref{eq:cPEAG_convergence1} by using $w^k := Fx^k + \xi^k \in Fx^k + Tx^k$.
\revised{The last bound \eqref{eq:cPEAG_convergence1b} is a direct consequence of \eqref{eq:cPEAG_convergence1} and $\norms{\Gc_{\eta}x^k} \leq \norms{Fx^k + \xi^k}$ in Lemma~\ref{le:P_lowerbound}.}
\Eproof
\end{proof}
%%% End of the proof.

\beforesec
\section{Concluding Remarks}\label{sec:concl_remarks}
\aftersec
In this paper, we have developed two ``Nesterov's accelerated'' variants of the FBFS method and its past-FBFS scheme.
Our approach is different from Halpern's fixed-point iteration in, e.g.,  \cite{cai2022accelerated,cai2022baccelerated,lee2021fast,tran2021halpern,yoon2021accelerated}.
These two types of methods are though different, they are related to each other as discussed in \cite{tran2022connection}.
We have established $\BigO{1/k}$ last-iterate convergence rates of both new schemes using explicit parameter update rules and different Lyapunov functions.
We have also provided an alternative convergence rate analysis of the extra-anchored gradient-type methods in \cite{cai2022accelerated,cai2022baccelerated} for solving \eqref{eq:coMI}.
Our analysis exploits the proof techniques in \cite{lee2021fast,tran2021halpern,yoon2021accelerated} and our second scheme is also different from the one in \cite{cai2022baccelerated}.

Beyond the results of this paper, several research questions remain open to us.
We discuss some of them here. 
First, though both proposed algorithms theoretically achieve better convergence rates than the classical non-accelerated \eqref{eq:EG} and \eqref{eq:PEG} schemes, their practical performance remains unclear.
It is important to develop practical variants of these methods so that they can perform better or at least not worse than  \eqref{eq:EG} and \eqref{eq:PEG} on concrete applications.
One idea is to use adaptive techniques to update the algorithmic parameters without knowing $L$ and $\rho$.
Second, it is also interesting to establish $\SmallO{1/k}$ (small-o) convergence rates of the proposed methods by using different parameter update rules, at least for the monotone case.
In addition, the convergence  of iterate sequences and their convergence rates remains open.
Third, improving the range of $L\rho$ is also interesting, at least theoretically.
The classical methods  \eqref{eq:EG} and \eqref{eq:PEG} have tight bounds on the range of $L$ and $\rho$, while our conditions remain suboptimal and can be improved, especially for \eqref{eq:AcPEG}. 
Fourth, one can also extend our methods to solve $0 \in Fx + Hx + Tx$ and provide rigorous convergence rate analysis as we have done in this paper.
Finally, developing stochastic and randomized block coordinate variants for these methods to solve \eqref{eq:coMI} is also interesting and expected to make significant impact in machine learning applications and distributed systems.

{
\vspace{2ex}
\noindent\textbf{Data availability.}
The author confirms that this paper does not contain any data.
%all data used in this paper is generated synthetically.
%The method for generating data is also described in the paper.
}

\vspace{2ex}
\noindent\textbf{Acknowledgements.}
The author is grateful to the anonymous reviewers for their helpful comments and suggestions.
This paper is based upon work partially supported by the National Science Foundation (NSF), grant no. NSF-RTG DMS-2134107 and the Office of Naval Research (ONR), grant No. N00014-20-1-2088 and No. N00014-23-1-2588.

%%% References
\bibliographystyle{plain}
%\bibliography{/Users/quoctd/Dropbox/E-Books/tran_bibtex_new}

\end{document}